\documentclass{amsart}

\usepackage{amsmath,amsthm,amssymb,fancyhdr,graphicx,bbm,cancel,color,mathrsfs,todonotes,hyperref,xypic, pinlabel}
\usepackage[all]{xy}

\newtheorem{thm}{Theorem}[section]
\newtheorem{lem}[thm]{Lemma}
\newtheorem{prop}[thm]{Proposition}
\newtheorem{cor}[thm]{Corollary}
\theoremstyle{definition} %The commands below have bold title, standard text

\newtheorem{qu}[thm]{Question} 
 
\theoremstyle{remark} %The commands below have italic title, standard text
\newtheorem{rmk}[thm]{Remark}
\numberwithin{equation}{section} %I don't know what this does!!!

\newcommand{\re}{\mathbb{R}}\newcommand{\q}{\mathbb{Q}}
\newcommand{\co}{\mathbb{C}}\newcommand{\z}{\mathbb{Z}}
\newcommand{\na}{\mathbb N}

\newcommand{\al}{\alpha}\newcommand{\de}{\delta}\newcommand{\ep}{\epsilon}\newcommand{\si}{\sigma}
\newcommand{\vp}{\varphi}\newcommand{\De}{\Delta}\newcommand{\la}{\lambda}
\newcommand{\Ga}{\Gamma}
\newcommand{\wtil}{\widetilde}\newcommand{\ta}{\theta}\newcommand{\om}{\omega}

\newcommand{\cd}{\cdots}\newcommand{\ld}{\ldots}
\newcommand{\sbs}{\subset}\newcommand{\bs}{\backslash}\newcommand{\pa}{\partial}

\newcommand{\xra}{\xrightarrow}

\newcommand{\ra}{\rightarrow}
\newcommand{\hra}{\hookrightarrow}\newcommand{\onto}{\twoheadrightarrow}

\newcommand{\ca}[1]{\mathcal{#1}}\newcommand{\un}[1]{\underline{#1}}\newcommand{\mf}{\mathfrak}

\newcommand{\fr}[2]{\frac{#1}{#2}}

\newcommand{\ot}{\otimes}
\newcommand{\lan}{\langle}\newcommand{\ran}{\rangle}
\newcommand{\op}{\oplus}
\newcommand{\ti}{\times}

\newcommand{\aut}{\text{Aut}}

\newcommand{\rest}[2]{#1\bigr\vert_{#2}}

\newcommand{\Sl}{\text{SL}}
\newcommand{\SO}{\text{SO}}

\newcommand{\homeo}{\text{Homeo}}
\newcommand{\SU}{\text{SU}}
\newcommand{\SP}{\text{Sp}}

\newcommand{\Spin}{\text{Spin}}

\newcommand{\im}{\text{Im}}
\newcommand{\Hom}{\text{Hom}}

\newcommand{\gl}{\mbox{GL}}
\newcommand{\sym}{\mbox{Sym}}

\usepackage{subfig}
\addtocontents{toc}{\protect\setcounter{tocdepth}{1}}

\begin{document}

\title{Pontryagin classes of locally symmetric manifolds}

\author{Bena Tshishiku}
\address{Department of Mathematics, University of Chicago, Chicago, IL 60615} \email{tshishikub@math.uchicago.edu}

%\subjclass[2000]{Primary 54C40, 14E20; Secondary 46E25, 20C20}%    General info

\date{\today}


\begin{abstract}
In this note we compute low degree rational Pontryagin classes for every closed locally symmetric manifold of noncompact type. In particular, we answer the question: Which locally symmetric $M$ have at least one nonzero Pontryagin class?
\end{abstract}

\maketitle

\tableofcontents

\section{Introduction}\label{sec:intro}

For a manifold $M$ and $i>0$, the Pontryagin class $p_i(M)\in H^{4i}(M;\q)$ is a diffeomorphism invariant. When these classes are nonzero, they can serve as obstructions to certain geometric problems (see for example \cite{tshishiku1}). 

In this paper, we are interested in locally symmetric manifolds $M$ of noncompact type. Let $G$ be a semisimple Lie group without compact factors; let $K\sbs G$ a maximal compact subgroup; and let $\Ga\sbs G$ a cocompact, torsion-free lattice. The manifold $G/K$ has a $G$-invariant Riemannian metric and is a symmetric space of noncompact type. $\Ga$ acts freely and properly on $G/K$, and the closed manifold $M=\Ga\bs G/K$ is a locally symmetric manifold of noncompact type. 

%\vspace{.1in}
%\noindent{\bf Question.} For which $\Ga\bs G/K$ is $p_i(\Ga\bs G/K)\neq0$ for some $i>0$?
%\vspace{.1in} 

\begin{qu}\label{q:main}For which $\Ga\bs G/K$ is $p_i(\Ga\bs G/K)\neq0$ for some $i>0$?\end{qu}

\noindent Throughout this paper all cohomology groups $H^*(\cdot)$ will be with $\q$ coefficients unless otherwise specified. 

The classical approach to determine if $p_i(\Ga\bs G/K)\neq0$ is roughly as follows. Let $U\sbs G_\co$ be the maximal compact subgroup of the complexification of $G$. By the Proportionality Principle (see \cite{kt_flat}, Section 4.14), $p_i(\Ga\bs G/K)\neq0$ if and only if $p_i(U/K)\neq0$. In \cite{bh}, Borel-Hirzebruch relate $p_i(U/K)$ to the weights of the action of $K$ on $\text{Lie}(U)$. From this, showing $p_i(U/K)\neq0$ reduces to showing that a polynomial is nonzero modulo an ideal (see \cite{bh} for details). The paper \cite{bh} also contains explicit computations, including computations that $p_1(F_4/\Spin_9)\neq0$ and $p_1(G_2/\SO_4)\neq0$, and the first Chern class $c_1(U/K)\neq0$ when $U/K$ is a compact Hermitian symmetric space and $U$ is \emph{not} of exceptional type.% contains explicit computations of Chern classes---and hence also Pontryagin classes---for some $U/K$ (for example, the compact Hermitian symmetric spaces of nonexceptional type). 

\vspace{.1in}
\noindent The objective of this paper is three-fold:
\begin{enumerate}\vspace{.05in}
\item Give a stream-lined way to determine if $p_i(\Ga\bs G/K)\neq0$. One feature of our approach is that it does not use the Proportionality Principle mentioned above to reduce to computations for the compact dual $U/K$. Our main idea is to use the action of $\Ga$ on the visual boundary $\pa(G/K)$ to study Pontryagin classes. This differs from the approach of Borel-Hirzebruch but ultimately reduces to the same problem: determining if a polynomial is nonzero modulo an ideal.% as mentioned above.% reduces to similar computations. 
\vspace{.1in}
\item Compute low dimensional Pontryagin classes $p_i(\Ga\bs G/K)$ for every locally symmetric manifold of noncompact type, including every exceptional example. As mentioned above, for a handful of $G$ these computations follow from computations done by Borel-Hirzebruch in \cite{bh}. Computations for $G=E_{6(6)}$ and $G=E_{6(-26)}$ were done by Takeuchi \cite{takeuchi}; however, the author was unable to find computations for every semisimple Lie group, especially not in a single source. The purpose of the present paper is to have a single source for computations for every $\Ga\bs G/K$. 
\vspace{.1in}
\item Answer the question: For which $\Ga\sbs G$ is $p_i(\Ga\bs G/K)\neq0$ for some $i$? A priori the answer depends on the choice of both $\Ga$ and $G$, but it follows from the Proportionality Principle that the answer is independent of $\Ga$. The author's interest in this question arose from the work \cite{tshishiku1}, in which Theorem \ref{thm:main} below is applied to a Nielsen realization question. 
\end{enumerate}

%Let $G$ be a Lie group, and $a\in H^*(G)$ a cohomology class. It is an interesting question to ask whether $a$ comes from some bigger group. 

%Let $\Ga\sbs G$ be a lattice and consider a cohomology class $a\in H^*(\Ga;\q)$. It is an interesting to ask whether $a$ is the restriction of

\vspace{.1in}

\noindent{\bf Main results.} A complete list of the simple real Lie groups of noncompact type are contained in Tables 1 and 2 below. We use subscripts instead of parentheses wherever possible; for example, we use $O_n$ instead of $O(n)$ to denote the orthogonal group. %$O_n\sbs\gl_n(\re)$,  $U_n\sbs\gl_n(\co)$, and $\SP_n\sbs\gl_{2n}(\co)$ to denote the orthogonal, unitary, and compact symplectic group, respectively. 
The examples in Table 1 are complex Lie groups, viewed here as real Lie groups. For any locally symmetric space $M$ of noncompact type, the universal cover $\wtil M$ is a symmetric space. Up to isogeny the isometry group of $\wtil M$ has identity component equal to a product of groups in Tables 1 and 2. See \cite{helgason} for further information. 

\begin{table}[h!]
\parbox{.45\linewidth}{
\centering
\begin{tabular}{c|cc}
$G$&$K$\\\hline
$\Sl_n\co$&$ \SU_n$\\[1mm]
$\SO_n\co$&$ \SO_n$\\[1mm]
$\SP_{2n}\co$&$ \SP_n$\\[1mm]
$E_6^\co$&$E_6$\\[1mm]
$E_7^\co$&$E_7$\\[1mm]
$E_8^\co$&$E_8$\\[1mm]
$F_4^\co$&$F_4$\\[1mm]
$G_2^\co$&$G_2$\\[1mm]
\hline
\end{tabular}
\vspace{.1in}
\caption{Complex noncompact simple Lie groups}
}
\hfill
\parbox{.45\linewidth}{
\begin{tabular}{cc|c}
&$G$&$K$\\[1mm]\hline
\text{AI}&$\Sl_n\re$ & $\SO_n$\\[1mm]
\text{AII}&$\SU^*_{2n}$ & $\SP_n$ \\[1mm]
\text{AIII}&$\SU_{p,q}$ & $S(U_p\ti U_q)$ \\[1mm]
\text{BDI}&$\SO_{p,q}$ & $\SO_p\ti \SO_q$  \\[1mm]
\text{DIII}&$\SO^*_{2n}$&$U_n$\\[1mm]
\text{CI}&$\SP_{2n}\re$ & $U_n$ \\[1mm]
\text{CII}&$\SP_{p,q}$ & $\SP_p\ti\SP_q$ \\[1mm]
\text{E$_6$I}&$E_{6(6)}$& $\mf{sp}_4$ \\[1mm]
\text{E$_6$II}&$E_{6(2)}$& $\mf{su}_6\ti\mf{su}_2$ \\[1mm]
\text{E$_6$III}&$E_{6(-14)}$&$\mf{so}_{10}\ti\mf{so}_2$ \\[1mm]
\text{E$_6$IV}&$E_{6(-26)}$& $\mf f_4$ \\[1mm]
\text{E$_7$V}&$E_{7(7)}$& $\mf{su}_8$ \\[1mm]
\text{E$_7$VI}&$E_{7(-5)}$& $\mf{so}_{12}\ti\mf{su}_2$ \\[1mm]
\text{E$_7$VII}&$E_{7(-25)}$& $\mf e_6\ti\mf{so}_2$ \\[1mm]
\text{E$_8$VIII}&$E_{8(8)}$& $\mf{so}_{16}$ \\[1mm]
\text{E$_8$IX}&$E_{8(-24)}$& $\mf e_7\ti\mf{su}_2$ \\[1mm]
\text{F$_4$I}&$F_{4(4)}$& $\mf{sp}_3\ti\mf{su}_2$ \\[1mm]
\text{F$_4$II}&$F_{4(-20)}$& $\mf{so}_9$ \\[1mm]
\text{G$_2$}&$G_{2(2)}$& $\mf{su}_2\ti\mf{su}_2$ \\[1mm]
\hline\end{tabular}
\vspace{.1in}
\caption{Real noncompact simple Lie groups}
}
\end{table}

We have separated the simple, real Lie groups into two separate tables because for $G$ in Table 1, we can conclude $p_i(\Ga\bs G/K)=0$ for all $i>0$ without any computation (see Section \ref{sec:cw}). The main work involved in the present paper is to determine those $G$ in Table 2 for which $p_i(\Ga\bs G/K)=0$ for all $i>0$. Here is a summary of our results.

\begin{thm}\label{thm:main}
Let $G$ be any real, simple, noncompact Lie group and let $\Ga\sbs G$ be a cocompact lattice. Then $p_i(\Ga\bs G/K)=0$ for all $i>0$ if and only if $G$ is
\begin{itemize}
\item[(i)] one of the Lie groups in Table $1$, or 
\item[(ii)] one of $\emph{\Sl}_n(\re)$, $\emph{\SU}^*_{2n}$, $\emph{\SO}_{p,1}$, $\emph{\SO}_{2,2}$, $\emph{\SO}_{3,3}$, or $E_{6(-26)}$.  
\end{itemize}
\end{thm}

Theorem \ref{thm:main} and its proof (in Section \ref{sec:computation}) show that the answer to Question \ref{q:main} is somewhat subtle. For example, let $G=\SO_{p,q}$ and assume $p\ge q$. If $p,q\ge2$, then $p_1(\Ga\bs G/K)\neq0$ if and only if $p\neq q$. If $p=q$, then $p_2(\Ga\bs G/K)\neq0$ as long as $p\ge4$. If either $p=q=2$ or $p=q=3$ or $p>q=1$, then $p_i(\Ga\bs G/K)=0$ for all $i>0$.

%$If $p> q\ge2$, then $p_1(\Ga\bs G/K)\neq0$. If $p=q\ge4$, then $p_1(\Ga\bs G/K)=0$ but $p_2(\Ga\bs G/K)\neq0$. If either $p=q=2$ or $p=q=3$ or $p>q=1$, then $p_i(\Ga\bs G/K)=0$ for all $i>0$.

It turns out that if $G$ is not one of the groups from (i) or (ii) in Theorem \ref{thm:main}, and $\Ga\sbs G$ is cocompact, then either $p_1(\Ga\bs G/K)\neq0$ or $p_2(\Ga\bs G/K)\neq0$. Thus to answer Question \ref{q:main}, we need only consider low dimensional Pontryagin classes. To determine if $p_i(\Ga\bs G/K)\neq0$ for $i\ge3$ using the methods of this paper would be feasible, but more computationally intensive. 

%There are two parts to proving Theorem \ref{thm:main}. We must show that if $G$ one of the groups from (i) or (ii) in Theorem \ref{thm:main}, then $p_i(\Ga\bs G/K)=0$ for all $i>0$. In addition, we must show that if $G$ is not one of the groups from (i) or (ii) in Theorem \ref{thm:main}, then $p_i(\Ga\bs G/K)\neq0$ for some $i>0$. For the latter, we only need to consider $p_1(\Ga\bs G/K)$ and $p_2(\Ga\bs G/K)$. To determine if $p_i(\Ga\bs G/K)\neq0$ for $i\ge3$ using the methods of this paper would be feasible, but more computationally intensive. 

%\begin{thm}\label{thm:zero}Fix $n,p\ge2$. Let $G$ be either 
%\begin{itemize}
%\item[(i)] one of the Lie groups in Table $1$, or 
%\item[(ii)] one of $\emph{\Sl}_n(\re)$, $\emph{\SU}^*_{2n}$, $\emph{\SO}_{p,1}$, $\emph{\SO}_{2,2}$, $\emph{\SO}_{3,3}$, or $E_{6(-26)}$.  
%\end{itemize}
%Let $\Ga\sbs G$ be a cocompact lattice. Then the Pontryagin classes of $M=\Ga\bs G/K$ vanish identically.\end{thm}

%\begin{thm}\label{thm:nonzero}Let $G$ be one of the Lie groups in Table $2$, and let $\Ga\sbs G$ be a cocompact lattice. If $G$ is not one of the groups in in Theorem \ref{thm:zero}  $(ii)$, then $M=\Ga\bs G/K$ has a nonzero Pontryagin class. \end{thm}

The proof of Theorem \ref{thm:main} extends to give a complete classification of which $\Ga\bs G/K$ have $p_i(\Ga\bs G/K)=0$ for all $i>0$. %a nonzero Pontryagin class as follows. 
\begin{cor}\label{cor:nonzero}
Let $G=\prod G_i$ be a semisimple Lie group with simple, noncompact factors $G_i$, and let $\Ga\sbs G$ be a cocompact lattice. Then $p_i(\Ga\bs G/K)=0$ for all $i>0$ if and only if each $G_i$ is one of the groups from $(i)$ or $(ii)$ in Theorem \ref{thm:main}.%$M=\Ga\bs G/K$ has a nonzero Pontryagin class if and only if one of the $G_i$ is one of the groups from Theorem \ref{thm:nonzero}.
\end{cor}

Instead of asking for $\Ga\bs G/K$ with a nonzero Pontryagin \emph{class}, one could ask for $\Ga\bs G/K$ with a nonzero Pontryagin \emph{number}. Specifically, let $M$ be a manifold with $\dim M=4k$, and let $i_1,\ld,i_m\in\na$ such that $i_1+\cd+i_m=k$. The cup product $p_{i_1}(M)\cup\cd\cup p_{i_m}(M)$ can be evaluated on the fundamental class $[M]\in H_n(M)$, and the resulting integer 
\[\big\lan p_{i_1}(M)\cup\cd\cup p_{i_m}(M),[M]\big\ran\in\z\]
is called a \emph{Pontryagin number}. These integers are \emph{topological} invariants of $M$ by Novikov's theorem (see \cite{lafont}, for example). We remark that a manifold can have zero Pontryagin numbers but have some nonzero Pontryagin classes. For example, take $\Ga\sbs G=\SO_{p,q}$ with $p,q$ both odd. As explained in \cite{lafont} Theorem B, the Pontryagin numbers of $\Ga\bs G/K$ are all zero. On the other hand, we show that if $p>q>2$, then $p_1(\Ga\bs G/K)\neq0$ (see Section \ref{sec:computation}). For more information about Pontryagin numbers of locally symmetric spaces, see \cite{lafont}. For recent results about the Euler characteristic of homogeneous spaces, see Mostow \cite{mostow}.

\vspace{.1in}

\noindent {\bf Method of proof.} We compute $p_i(\Ga\bs G/K)$ by the following procedure. Let $n=\dim M$. The unit tangent bundle $T^1M\ra M$ has a flat $\homeo(S^{n-1})$ structure with monodromy $\mu:\Ga\ra\homeo(S^{n-1})$ given by the action of $\Ga$ on the visual boundary $\pa(G/K)\simeq S^{n-1}$. The homomorphism $\mu$ induces a map of classifying spaces $M\approx B\Ga\ra B\homeo(S^{n-1})$ and hence a map 
\[\mu^*: H^*\big(B\homeo(S^{n-1})\big)\ra H^*(M).\] 
%In Section \ref{sec:algorithm} we describe 
There are classes $q_i\in H^{4i}\big(B\homeo(S^{n-1})\big)$ for which $\mu^*(q_i)=p_i(M)$ (see Section \ref{sec:algorithm}). 
%Under $\mu^*$ the Pontryagin classes in $H^*\big(B\homeo(S^{n-1})\big)$ map to the Pontryagin classes of $M$. 
To determine if $\mu^*(q_i)=0$, note that $\mu$ factors $\mu=\al_1\circ\al_2\circ\al_3$, where
\begin{equation}\label{eqn:rep}
\Ga\xra{\al_3} G^\de\xra{\al_2} G\xra{\al_1} \homeo(S^{n-1}).\end{equation}
Here $G^\de$ is the abstract group $G$ viewed as a Lie group with the discrete topology. The map $\al_3$ is the inclusion, $\al_2$ is the identity (which is continuous),  and $\al_1$ is the action of $G$ on its visual boundary. 

To understand $\mu^*=(\al_1\circ\al_2\circ\al_3)^*$, we study the individual maps 
\vspace{.05in}
\begin{equation}\label{eqn:main}\boxed{H^*\big(B\homeo(S^{n-1})\big)\xra{\al_1^*}H^*(BG)\xra{\al_2^*}H^*(BG^\de)\xra{\al_3^*}H^*(M).}\end{equation}

\vspace{.05in}

\noindent\un{Step 1.} Computing $\al_1^*$ reduces to computing weights of the adjoint action of $K$ on $\mf g=\text{Lie}(G)$ (see Section \ref{sec:isotropy1}). 

\vspace{.1in} 

\noindent\un{Step 2.} Computing $\al_2^*$ reduces to computing the map $H^*(BG_\co)\ra H^*(BG)$ induced by complexification $G\ra G_\co$. This follows from Chern-Weil theory (see Section \ref{sec:cw}). 

\vspace{.1in}

\noindent\un{Step 3.} $\al_3^*$ is injective on the image of $H^*(BG)\ra H^*(BG^\de)$ by a transfer argument (see Section \ref{sec:computation}).

\vspace{.1in}

\noindent{\bf Structure of the paper.} In Section \ref{sec:algorithm} we define the classes $q_i$ mentioned above and show that $\mu^*(q_i)=p_i(M)$. 
%explain why the procedure outlined above computes Pontryagin classes. 
In Section \ref{sec:background} we recall Borel's computation of $H^*(BK)$ for $K$ a compact Lie group, and we recall how characteristic classes of a representation can be computed in terms of the weights of that representation. In Sections \ref{sec:isotropy1} and \ref{sec:cw} we complete Steps 1 and 2, respectively. %compute the maps $\al_1^*$ and $\al_2^*$, respectively. 
In Section \ref{sec:computation} we explain Step 3 and combine Steps 1, 2, and 3 to conclude which $\Ga\bs G/K$ have $p_i(\Ga\bs G/K)\neq0$ for some $i>0$ and thus prove Theorem \ref{thm:main}. 
%the computations of $\al_1^*$, $\al_2^*$, and $\al_3^*$ to compute Pontryagin classes of $\Ga\bs G/K$. 

\subsection{Acknowledgements} The author would like to thank his advisor B.\ Farb for his gracious and ceaseless support, for his encouragement to complete this project, and for extensive comments that significantly improved a draft of this paper. %making many changes to the organization and readability of the paper

\section{An algorithm for computing Pontryagin classes}\label{sec:algorithm}

%In this section we establish that $\mu^*(q_i)=p_i(M)$. 
%the procedure described in the introduction actually computes the Pontryagin classes of $M=\Ga\bs G/K$. 
This section has two goals. First we recall the definition of the Pontryagin classes of topological sphere bundles $q_i\in H^{4i}\big(B\homeo(S^{n-1})\big)$. Then we explain why $\mu^*(q_i)=p_i(M)$. This will follow from the construction of a flat $\homeo(S^{n-1})$ structure on the unit tangent bundle $T^1M\ra M$. %the unit tangent bundle has a flat $\homeo(S^{n-1})$ structure. and why the monodromy factors as in (\ref{eqn:rep}). 

\subsection{Pontryagin classes of sphere bundles} The Pontryagin classes $p_i\in H^*(BO_n)$ are classically defined as invariants of real vector bundles (see \cite{ms}). The following proposition shows that these invariants can also be defined for topological $\re^n$-bundles. 
\begin{prop}\label{prop:ratsurj}
 The inclusion $g:O_{n}\hra\emph{\homeo}(\re^n)$ induces a surjection  
\[g^*:H^*\big(B\emph{\homeo}(\re^n);\q\big)\ra H^*\big(BO_{n};\q\big).\]
%with rational coefficients.
\end{prop}

Proposition \ref{prop:ratsurj} can be proved using results from Kirby-Siebenmann \cite{ks}. The argument (which the author learned from A. Hatcher) is given in \cite{tshishiku1}. 

From Proposition \ref{prop:ratsurj}, Pontryagin classes of sphere bundles can be defined as follows. Define a homomorphism $\de:\homeo(S^{n-1})\ra\homeo(\re^{n})$ using the Alexander trick: $\de(f)$ performs the homeomorphism $f$ on the sphere of radius $r$ for every $r>0$, and $\de(f)$ fixes the origin. This induces maps between classifying spaces and hence a map
\[\de^*:H^*\big(B\homeo(\re^{n})\big)\ra H^*\big(B\homeo(S^{n-1})\big)\]
The restriction of $\de$ to the subgroup $O_n\sbs\homeo(S^{n-1})$ is the standard action $O_n\ra\homeo(\re^n)$, so there is a commutative diagram
\begin{equation}\label{diag:ratsurj}
\begin{gathered}
\begin{xy}
(-22,0)*+{H^*\big(B\homeo(\re^n)\big)}="A";
(22,0)*+{H^*\big(B\homeo(S^{n-1})\big)}="B";
(0,-10)*+{H^*(BO_n)}="C";
{\ar"A";"B"}?*!/_3mm/{\de^*};
{\ar "A";"C"}?*!/^3mm/{g^*};
{\ar "B";"C"}?*!/_3mm/{r^*};
\end{xy}
\end{gathered}
\end{equation}

By Proposition \ref{prop:ratsurj}, there is a class $\wtil p_i\in H^{4i}\big(B\homeo(\re^n)\big)$ with $g^*(\wtil p_i)=p_i$. Since Diagram \ref{diag:ratsurj} commutes, $\de^*(\wtil p_i)\in H^{4i}\big(B\homeo(S^{n-1})\big)$ is nontrivial. We refer to the classes $q_i=\de^*(\wtil p_i)$ as the \emph{Pontryagin classes of topological sphere bundles}.

\subsection{Flat structure on the unit tangent bundle} %We explain why $T^1M\ra M$ has a flat $\homeo(S^{n-1})$ structure and explain why the monodromy $\mu:\Ga\ra\homeo(S^{n-1})$ factors as described in the introduction. 
We continue to assume that $G$ is a semisimple Lie group without compact factors and that $K\sbs G$ is a maximal compact subgroup. With these assumptions $G/K$ is contractible and has a metric of nonpositive curvature so that $G$ acts by isometries on $G/K$. In addition, $G$ acts on the visual boundary $\pa (G/K)\simeq S^{n-1}$ (see e.g.\ \cite{bgs}). If $G$ has rank 1, then the action of $G$ on $\pa(G/K)$ is smooth, but this is not known in general. Thus, even though $G/K$ is an algebraic example of a contractible, nonpositively curved manifold, the action on $\pa(G/K)$ is a priori only an action by homeomorphisms. By restriction to $\Ga\sbs G$, we obtain an action $\Ga\ra\homeo(S^{n-1})$, and from this action we can build an $S^{n-1}$-bundle $E\ra M$, where $E$ is the quotient of $(G/K)\ti S^{n-1}$ by the diagonal action of $\Ga$. %In fact, $E\ra M$ is isomorphic to the unit tangent bundle. 

\begin{lem}[See e.g.\ \cite{tshishiku1}]\label{lem:tang}
Let $M$ be a complete Riemannian manifold of nonpositive curvature with universal cover $\wtil M$. The sphere bundle with monodromy given by the action of the deck group $\pi_1(M)$ on the visual boundary $\pa\wtil M\simeq S^{n-1}$ is isomorphic to the unit tangent bundle $T^1M\ra M$. 
\end{lem}

Since $T^1M\ra M$ is flat with monodromy $\mu:\Ga\ra\homeo(S^{n-1})$ equal to the $\Ga$-action on $\pa (G/K)$, which is the restriction of the $G$-action on $\pa (G/K)$, the monodromy factors as claimed in (\ref{eqn:rep}).

\section{Compact Lie groups, characteristic classes, and representations}\label{sec:background}

Let $K$ be a compact Lie group and let $BK$ be its classifying space. In this section we recall Borel's computation of the cohomology $H^*(BK)$ (see Theorem \ref{thm:borel}). For a representation $\rho:K\ra\gl_m(\co)$, we recall how Theorem \ref{thm:borel} allows one to compute the image of $\rho^*:H^*(B\gl_m(\co))\ra H^*(BK)$ (specifically the image of the Chern classes) in terms of the weights of $\rho$. For more details, see \cite{bh}.

\subsection{The cohomology of $BK$}\label{sec:cohoBK}

Let $K$ be a compact Lie group. Let $S\sbs K$ be a maximal abelian subgroup. $S$ is homeomorphic to an $r$-torus $(S^1)^r$ for some integer $r$, which is called the \emph{rank} of $K$. Let $N_K(S)$ denote the normalizer of $S$ in $K$. The \emph{Weyl group} $W$ is defined as $N_K(S)/S$. 

\begin{thm}[Borel \cite{borel:toplie}]\label{thm:borel} Let $K$ be a compact Lie group with maximal torus $S$ and Weyl group $W$. The inclusion $S\hra K$ induces an isomorphism 
\[H^*(BK;\q)\simeq H^*(BS;\q)^W.\]
\end{thm}

\subsection{Weights, transgression, and characteristic classes}\label{sec:chern}
Let $K$ be a compact Lie group with maximal torus $S$, and let $\rho:K\ra\gl_n(\co)$ be a representation. The restriction $\rho\rest{}{S}$ is a sum of 1-dimensional representations $\la_i:S\ra\co^\ti$ called the \emph{weights} of $\rho$. By the identification $\co^\ti=K(\z,1)$ and the fact that cohomology is a represented functor, the weights can be viewed as elements of $H^1(S;\z)$. Since we are interested in cohomology with $\q$-coefficients, we view the weights as elements of $H^1(S)=H^1(S;\q)$. 

For the representation $\rho$, we are interested in computing the induced map 
\[\rho^*:H^*\big(B\gl_n(\co);\q\big)\ra H^*(BK;\q)\]
in terms of the weights $\la_1,\ld,\la_n\in H^1(S;\q)$. Let $S\ra ES\ra BS$ be the universal $S$-bundle. The edge map 
\[\tau:H^1(S;\q)\ra H^2(BS;\q)\]
on the $E_2$-page of the Serre spectral sequence of this fibration is an isomorphism, called the \emph{transgression}. Let $w_i=-\tau(\la_i)$. The \emph{total Chern class} $c(\rho)\in H^*(BS;\q)$ is defined by
\begin{equation}\label{eqn:totalchern}c(\rho)=1+c_1(\rho)+\cd+c_n(\rho)=\prod_{i=1}^n(1+w_i).\end{equation}
The Weyl group $W$ permutes the weights of $\rho$, so $c(\rho)\in H^*(BS)^W\simeq H^*(BK)$. 

If $\rho:K\ra\gl_n(\re)$ is a real representation and $\rho_\co:K\ra\gl_n(\co)$ is the complexification, one defines the $i$-th \emph{Pontryagin class} $p_i(\rho)$ of $\rho$  by the formula
\[p_i(\rho)=(-1)^{i}c_{2i}(\rho_\co).\]

\vspace{.05in}
\begin{rmk} The transgression $\tau$ can be given concretely as follows. 
Identify $H^1(H;\z)\simeq\Hom(H,\co^\ti)$, and define $\Hom(S,\co^\ti)\ra H^2(BS)$: Given $\vp:S\ra\co^\ti$, form the space
\[E_\vp=\fr{EH\ti \co^\ti}{H}\]
which is the quotient  of $EH\ti \co^\ti$ by the diagonal action of $H$, where $H$ acts on $\co^\ti$ by $\vp$. Now $E_\vp$ has a natural projection to $EH/H=BH$, and this makes $E_\vp$ a $\co^\ti$ bundle over $BH$. The first Chern class $c_1(E_\vp)$ lives in $H^2(BS;\z)$, and the transgression is given by $\tau(\vp)=c_1(E_\vp)$. 
\end{rmk}

\subsection{The invariant polynomials $H^*(BS)^W$}\label{sec:invar} Let $K$ be a compact Lie group with maximal torus $S$. The ring of invariant polynomials $H^*(BS)^W$ is well-known (see \cite{humphreys}, Ch.\ 3). For the exceptional groups, explicit polynomial generators for $H^*(BS)^W$ can be computed as follows. Let $K$ be one of exceptional compact Lie groups: $E_8, E_7, E_6, F_4, G_2$. Let $V$ be a fundamental representation of $K$ of minimal dimension. Denote $\la_1,\ld,\la_d\in H^1(S)$ the weights of this representation, so that $\tau(\la_i)\in H^2(BS)$. In \cite{mehta}, Mehta shows that power sums 
\[I_k=\sum_{i=1}^d \tau(\la_i)^k\] generate the invariant polynomials $H^*(BS)^W$. %For $K=E_6$ and $E_7$ we describe these polynomials with respect to different embeddings $K\hra E_8$. 
In the remainder of this section we recall the descriptions of $H^*(BS)^W$ for the different compact Lie groups, and we record explicit generators of $H^*(BS)^W$ that will be used in Sections \ref{sec:cw} and \ref{sec:computation}. 

To express the weights $\la_1,\ld,\la_d$ for the exceptional $K$, we use the descriptions from Adams \cite{adams}. We remark that our expressions for $I_k$ for $E_8, E_7, E_6,$ and $F_4$ agree with \cite{mehta} up to a chance of basis. Let $\sym(x_1,\ld,x_n)\sbs \q[x_1,\ld,x_n]$ denote the subring of symmetric polynomials. 

\vspace{.1in}

\noindent \un{$B\SO_n$}: let $k=\lfloor n/2\rfloor$. As a subring of $\q[y_1,\ld,y_k]$, 
\begin{equation*}%\label{eqn:cohoSO(n)}
%\begin{gathered}
H^*(B\SO_n)\simeq H^*(BS)^W=\left\{
\begin{array}{clll}
\sym(y_1^2,\ld,y_k^2)&n=2k+1\\[2mm]
\lan\sym(y_1^2,\ld,y_k^2), y_1\cd y_k\ran& n=2k
\end{array}\right.
%\end{gathered}
\end{equation*}

%\vspace{.1in}

\noindent\un{$BU_n$}:
\[H^*(BU_n)\simeq H^*(BS)^W=\sym(y_1,\ld,y_n).\]
Note that $H^*(B\SU_n)$ is the quotient of $\sym(y_1,\ld,y_n)$ by the ideal generated by $y_1+\cd+y_n$. Similarly, $H^*\big(BS(U_p\ti U_q)\big)$ is the quotient of 
\[\sym(y_1,\ld, y_p)\ot\sym(z_1,\ld,z_q)\] by the ideal generated by $y_1+\cd+y_p+z_1+\cd+z_q$. 

\vspace{.1in}

\noindent\un{$B\SP_n$}: $H^*\big(B\SP(n)\big)\simeq H^*(BS)^W=\sym(y_1^2,\ld,y_n^2)$.

\vspace{.1in}

\noindent\un{$BE_8$}: Let $S\sbs E_8$ be a maximal torus with Weyl group $W$. Up to conjugation, we can assume $S\sbs\Spin_{16}\sbs E_8$ and that $S$ is a maximal torus of $\Spin_{16}$. This allows us to identify $H^1(S)\simeq\q\{J_1,\ld,J_8\}$ and $H^2(BS)\simeq\q[z_1,\ld,z_8]$ in such a way that the roots of $E_8$ are 
\begin{equation}\label{eqn:rootsE8}
\begin{gathered}
\left\{\begin{array}{ccc}\pm J_i\pm J_j&1\le i<j\le 8,\\[1mm] \fr{1}{2}(\pm J_1\pm\cd\pm J_8)&\text{ even number of $-$'s}\end{array}\right.
\end{gathered}
\end{equation}
See \cite{adams} (pg.\ 56). $W$ preserves the roots, and $H^*(BS)^W\simeq H^*(BE_8)$ is generated by 
\[I_{2k}=\sum_{1\le i<j\le 8}(z_i+z_j)^{2k}+(z_i-z_j)^{2k}+\sum_{\text{even $-$'s}}\fr{1}{2^{2k}}(z_1\pm\cd\pm z_8)^{2k}.\]
Moreover, according to \cite{mehta} (pg.\ 1088), $H^*(BS)^W\simeq\q[I_2,I_8,I_{12}, I_{14}, I_{18}, I_{20}, I_{24}, I_{30}]$. We record here that $I_2=30(z_1^2+\cd+z_8^2)$. 

\vspace{.1in}

\noindent\un{$BE_7$}: For any embedding $\SU_2\hra E_8$, the identity component of the centralizer of $\SU_2$ is isomorphic to $E_7$ (see \cite{adams}, pg.\ 49). Choose $\SU_2$ so that its roots are $\pm(J_7-J_8)$ in $E_8$ (cf.\ \ref{eqn:rootsE8}). Let $S\sbs E_7$ be a maximal torus with Weyl group $W$. Since the roots of $E_7$ are orthogonal to the roots of $\SU_2$, we can identify $H^1(S)\simeq\q\{J_1,\ld,J_6, J_7+J_8\}$. Let $z_i=\tau(J_i)$ in $H^2(BS)$ for $1\le i\le 6$ and let $z_7=\tau(J_7+J_8)$. 

Following \cite{adams} (pg.\ 52), the fundamental representation of $E_7$ is 56-dimensional, and by restricting this representation to $\mf{so}_{12}\ti\mf{su}_2\sbs\mf e_7$, one can compute the weights: 
\[\left\{\begin{array}{lll}
\pm J_i\pm (\fr{J_7+J_8}{2})&1\le i\le 6\\[1mm]
\fr{1}{2}(\pm J_1\pm\cd\pm J_6)& \text{odd number of $J_i$ have $-$ sign}
\end{array}\right.\]
$W$ preserves these weights, and $H^*(BS)^W\simeq H^*(BE_7)$ is generated by 
\[I_{2k}=\sum_{i=1}^6(\fr{z_7}{2}+z_i)^{2k}+(\fr{z_7}{2}-z_i)^{2k}+\fr{1}{2^{2k}}\sum(z_1\pm\cd\pm z_6)^{2k},\]
where the second sum is over all terms with an odd number of $-$ signs. According to \cite{mehta} (pg.\ 1086), $H^*(BS)^W\simeq \q[I_2, I_6, I_8, I_{10}, I_{12}, I_{14}, I_{18}]$. Note that for this description of $E_7$, 
\[I_2=6(z_1^2+\cd+z_6^2)+3z_7^2.\]

\vspace{.1in} 

\noindent\un{A second computation for $BE_7$}: We will also need the following description of $H^*(BE_7)$. This time choose $\SU_2\hra E_8$ so that its roots are $\pm(J_1+\cd+J_8)$. Let $S\sbs E_7$ be a maximal torus with Weyl group $W$. $H^1(S)$ is the orthogonal complement of $J_1+\cd+J_8$ in $\q\{J_1,\ld,J_8\}$. One can compute the weights (with respect to this description of $H^1(S)$) of the fundamental representation of $E_7$ by restricting to $\SU_8\sbs E_7$ (see \cite{adams}, pg.\ 69). The weights are
\[\pm\left(J_i+J_j-\fr{1}{4}(J_1+\cd+J_8)\right),\>\>\>\>\>1\le i<j\le 8.\] 
Let $z_i=\tau(J_i)\in H^2(BS)$ for $1\le i\le 8$. (Technically $J_i$ does not live in $H^1(S)$, so we really mean the restriction of the weight $J_i$ for $E_8$ to a weight for $E_7$.) $H^*(BS)^W\simeq H^*(BE_7)$ is generated by 
\[I_{2k}=\sum_{1\le i<j\le 8}\left((z_i+z_j)-\fr{1}{4}(z_1+\cd+z_8)\right)^{2k}.\]
As above, $H^*(BS)^W=\q[I_2,I_6, I_8,I_{10},I_{12}, I_{14}, I_{18}]$, and one computes
\[I_2= \fr{3}{4}\left[7\left(\sum_{1\le i\le8} z_i^2\right)+2\left(\sum_{1\le i<j\le 8} z_iz_j\right)\right]\]

\vspace{.1in}
\noindent\un{$BE_6$}: For any embedding $\SU_3\hra E_8$, the identity component of the centralizer of $\SU_3$ is isomorphic to $E_6$ (see \cite{adams}, pg.\ 49). Choose $\SU_3$ so that its roots are $\pm (J_6-J_7)$, $\pm (J_6-J_8)$, and $\pm (J_7-J_8)$ (cf.\ \ref{eqn:rootsE8}). Let $S\sbs E_6$ be a maximal torus with Weyl group $W$. Since the roots of $E_6$ are orthogonal to the roots of $\SU_3$, we can identify 
\[H^1(S)\simeq\q\{J_1,\ld,J_5, J_6+J_7+J_8\}.\] Let $z_i=\tau(J_i)\in H^2(BS)$ for $1\le i\le 5$, and let $z_6=\tau(J_6+J_7+J_8)$. 

Lie$(E_6)=\mf e_6$ has two 27-dimensional fundamental representations $U_1,U_2$. One can compute the weights of $U_i$ by restricting to $\mf{so}_{10}\ti\mf{so}_2\sbs\mf e_6$. Following \cite{adams} (pg.\ 53), $U_1=\xi^{-4}+\la_{10}^1\ot\xi^2+\De^+\ot\xi^{-1}$ and its weights are
%the fundamental representation of $E_6$ is 27-dimensional, and by restricting this representation to $\mf{so}_{10}\ti\mf{so}_2\sbs\mf e_6$, one can compute the weights: 
\[\left\{\begin{array}{lll}
\pm J_i+\fr{1}{3}(J_6+J_7+J_8)& 1\le i\le 5\\[2mm]
\fr{1}{2}\left(\pm J_1\pm\cd\pm J_5-\fr{1}{3}(J_6+J_7+J_8)\right)&\text{even number of $J_i$ have $-$ sign}\\[2mm]
-\fr{2}{3}(J_6+J_7+J_8)
\end{array}\right.\]
$W$ preserves these weights, and $H^*(BS)^W\simeq H^*(BE_6)$ is generated by 
\begin{equation}\label{eqn:invarE6}I_k=\left(-\fr{2}{3}z_6\right)^k+\sum_{i=1}^5\left(\fr{1}{3}z_6+z_i\right)^k+\left(\fr{1}{3}z_6-z_i\right)^k+\fr{1}{2^k}\sum \left(\pm z_1\pm\cd\pm z_5-\fr{1}{3}z_6\right)^k.\end{equation}
In particular, $I_1=I_3=0$ and $I_2=6 (z_1^2 + z_2^2 + z_3^2 + z_4^2 + z_5^2) + 2 z_6^2$ and  $I_4=\fr{1}{12}(I_2)^2$. According to \cite{mehta} (pg.\ 1086), $H^*(BS)^W\simeq\q[I_2,I_5, I_6, I_8, I_9, I_{12}]$.

\vspace{.1in}

\noindent\un{A second computation for $BE_6$}: We will also need the following description of $H^*(BE_6)$. This time choose $\SU_3\hra E_8$ so that the roots are 
\begin{equation}\label{eqn:rootsE6}\left\{\pm\fr{J_1+\cd+J_6}{2}\pm\fr{J_7+J_8}{2}, 
\>\>\>\>\pm (J_7+J_8)\right\}.\end{equation}
Let $S\sbs E_6$ be a maximal torus with Weyl group $W$. $H^1(S)$ is the subspace of $\q\{J_1,\ld,J_8\}$ that is orthogonal to the roots in (\ref{eqn:rootsE6}). 

For this description of $\mf e_6$, we compute the weights of the fundamental representations $U_1,U_2$ by restricting to $\mf{su}_6\ti\mf{su}_2\sbs\mf e_6$. As a $(\mf{su}_6\ti\mf{su}_2)$-representation,
\[U_1=\la_6^2+\la_6^5\ot\la_2^1\>\>\>\>\text{ and }\>\>\>\>U_2=\la_6^4+\la_6^1\ot\la_2^1,\] 
where $\la_m^i$ denotes the $i$-exterior power of the standard representation of $\mf{su}_m$. Since a Cartan subalgebra for $\mf{su}_6\ti\mf{su}_2$ is also a Cartan subalgebra for $\mf e_6$, we can identify the weights of $U_i$ as a $\mf{su}_6\ti\mf{su}_2$-representation with the weights of $U_i$ as a $\mf e_6$-representation. 

Let $u=\fr{J_1+\cd+J_6}{2}+\fr{J_7+J_8}{2}$ and $v=\fr{J_1+\cd+J_6}{2}-\fr{J_7+J_8}{2}$. One computes that the weights of $U_1$ are 
\[\begin{array}{cccc}
J_i+J_j-\fr{1}{3}(u+v) &1\le i<j\le 6& \text{(corresponding to $\la_6^2$)}\\[2mm]
J_i\pm \fr{J_7-J_8}{2}-\fr{5}{6}(u+v) &1\le i\le 6& \text{(corresponding to $\la_6^5\ot\la_2^1$)}
\end{array}
\]
$W$ preserves these weights. Let $z_i=\tau(J_i)\in H^2(BS)$ for $1\le i\le 6$ and let $z_7=\tau(J_7-J_8)$. (As above, by $J_i$ we really mean the restriction of the weight for $E_8$ to a weight for $E_7$.) With respect to this basis $H^*(BS)^W\simeq H^*(BE_6)$ is generated by 
\[I_{k}= \sum_{1\le i<j\le 6}\left(z_i+z_j-\fr{1}{3}(z_1+\cd+z_6)\right)^k+\sum_{\substack{1\le i\le 6\\\ep\in\{\pm1\}}}\left(z_i-\fr{5}{6}(z_1+\cd+z_6)+\fr{\ep}{2}z_7
\right)^k\]
As above, $H^*(BS)^W\simeq\q[I_2,I_5, I_6, I_8, I_9, I_{12}]$, and one computes 
\[I_2=5(z_1^2+\cd+z_6^2)+3z_7^2-2\left(\sum_{1\le i<j\le 6}z_iz_j\right)
\]

\vspace{.1in}

\noindent\un{$BF_4$}: Let $S\sbs F_4$ be a maximal torus with Weyl group $W$. Up to conjugation we can assume $S\sbs \Spin_9\sbs F_4$ and that $S$ is a maximal torus of $\Spin_9$ (see \cite{adams}, pg.\ 53). This allows us to identify $H^1(S)\simeq\q\{L_1,\ld,L_4\}$ and $H^2(BS)\simeq\q[z_1,\ld,z_4]$. 
The 26-dimensional fundamental representation of $F_4$ has the following nonzero weights (these are the short roots of $F_4$, see \cite{adams} pg.\ 55). 
\[\left\{\begin{array}{cll}
\pm L_i&1\le i\le 4\\
\fr{1}{2}(\pm L_1\pm L_2\pm L_3\pm L_4)\end{array}\right.\]
$W$ preserves these weights, and $H^*(BS)^W\simeq H^*(BF_4)$ is generated by
\begin{equation}\label{eqn:invarF4}I_{2k}=\sum_{i=1}^4z_i^{2k}+\fr{1}{2^{2k}}\sum(\ep_1z_1+\ep_2z_2+\ep_3z_3+z_4)^{2k}.\end{equation}
The second sum is over all tuples $(\ep_1,\ep_2,\ep_3)\in\{\pm1\}^3$. According to \cite{mehta} (pg.\ 1091), $H^*(BS)^W=\q[I_2,I_6, I_8, I_{12}]$. Note that $I_2=3(z_1^2+z_2^2+z_3^2+z_4^2)$.

\vspace{.1in} 

\noindent\un{$BG_2$}: Let $S\sbs G_2$ be a maximal torus (its dimension is 2). Let $J_1,J_2\in H^1(S)$ be simple roots for $G_2$. Denote $z_i=\tau(J_i)\in H^2(BS)$. The Weyl group $W$ is the dihedral group of order 12 and it permutes the nonzero weights of the 7-dimensional fundamental representation. These weights are 
\[\big\{\pm J_1,\>\pm (J_1+J_2),\> \pm (2J_1+J_2)\big\}.\]
See \cite{fh} (Lecture 22), for example. $H^*(BS)^W$ is generated the polynomials 
\[I_{2k}=z_1^{2k}+(z_1+z_2)^{2k}+(2z_1+z_2)^{2k}.\]
According to \cite{mehta} (pg.\ 1094) $H^*(BS)^W\simeq\q[I_2,I_6]$. Note that  \[I_2=2(3z_1^2+3z_1z_2+z_2^2).\] 

\section{Computing $\al_1^*:H^*\big(B\homeo(S^{n-1})\big)\ra H^*(BG)$}\label{sec:isotropy1}

Since the inclusion $K\hra G$ is a homotopy equivalence, it induces an isomorphism $H^*(BG)\xra\sim H^*(BK)$. To understand $\al_1^*$, we study $\al_1\rest{}{K}:K\ra\homeo(S^{n-1})$. 

Let $T_{eK}^1(G/K)$ be the space of rays through the origin in $T_{eK}(G/K)$. The action of $K$ on $G/K$ induces an action on $T_{eK}^1(G/K)$. The exponential map defines a $K$-equivariant homeomorphism 
\[s: T_{eK}^1(G/K)\ra \pa(G/K).\]
%and the action of $K$ on $G/K$ induces $K$-actions on $T_{eK}^1(G/K)$ and on $\pa(G/K)$. It is easy to see that $s$ is equivariant with respect to these actions. 
The $K$-action on $T_{eK}(G/K)$ can be described as follows. The adjoint action of $K$ on $\mf g=\text{Lie}(G)$ decomposes into invariant subspaces $\mf k\op\mf p$, where $\mf k=\text{Lie}(K)$ and $\mf p\simeq T_{eK}(G/K)$. This implies the following lemma. For more details, see \cite{tshishiku1}.
%Since the conjugation action and the left action of $K$ on $G$ descend to the same action on $G/K$, the action of $K$ on $T_{eK}(G/K)\simeq\mf p$ is isomorphic to the adjoint action of  $K$ on $\mf p\sbs\mf k\op\mf p=\mf g$. Thus the action of $K$ on $\pa(G/K)$ is isomorphic to the action induced by $\iota: K\ra\aut(\mf p)$.

\begin{lem}\label{lem:linear}
The action of $K\sbs G$ on $\pa (G/K)$ is induced by a linear representation $\iota: K\ra\emph{\aut}(\mf p)$.
\end{lem}
We refer to the representation $\iota:K\ra\aut(\mf p)\simeq\gl_n(\re)$ as the \emph{isotropy representation}. Let $r:\gl_n(\re)\ra\homeo(S^{n-1})$ be the $\gl_n(\re)$-action on the space of rays through the origin in $\re^n$. By Lemma \ref{lem:linear}, $\al_1\rest{}{K}=r\circ\iota$.
%\[K\xra{\iota}\gl_n(\re)\xra{r}\homeo(S^{n-1}),\] where $r$ is the $\gl_n(\re)$-action on the space of rays through the origin in $\re^n$. 
Since the map 
\[r^*:H^*\big(B\homeo(S^{n-1})\big)\ra H^*(B\gl_n(\re))\]
is understood via Diagram (\ref{diag:ratsurj}), it remains only to understand
\[\iota^*:H^*\big(B\gl_n(\re)\big)\ra H^*(BK).\]
As described in Section \ref{sec:chern}, $\iota^*$ can be computed using the weights of $\iota$. This will be carried out in Sections \ref{sec:isotropySLn} $-$ \ref{sec:isotropyG2} as follows.  
%\vspace{.1in}
%\noindent{\bf Weights of the isotropy representation.} In Section \ref{sec:background} we described how Pontryagin classes of a representation $\rho$ can be computed from the weights of $\rho$. This will be applied here as follows. 
Let $S\sbs K$ be a maximal torus. The isotropy representation $\iota:K\ra\aut(\mf p)$ is a real representation, so we complexify to get a representation $\iota_\co:K_\co\ra\aut(\mf p_\co)$. On the maximal abelian subgroup $H\sbs K_\co$, we obtain weights $\la_i: H\ra\co^\ti$. To compute the weights explicitly, we pass to the Lie algebra $\mf h$ of $H$ and view the weights as elements of $\mf h^*=\Hom(\mf h,\co)$. After computing the weights of $\iota_\co$, we use Equation (\ref{eqn:totalchern}) to express the total Chern class $c(\iota_\co)$ as a polynomial in $H^*(BH)$. Finally, since $S\hra H$ induces an isomorphism $H^*(BH)\ra H^*(BS)$, we obtain the Pontryagin classes $p_i(\iota)=(-1)^ic_{2i}(\iota_\co)$ as polynomials in $H^*(BS)^W\simeq H^*(BK)\simeq H^*(BG)$. This computes $\al_1^*$ since $\al_1^*(q_i)=p_i(\iota)$. 

Below, $V_n$ will denote the standard representation of $\mf{gl}_n(\co)$, $\mf{sl}_n(\co)$, $\mf{so}_n(\co)$, or $\mf{sp}_{2k}(\co)$ (for $n=2k$). 

%The following notation will be used for the classical groups. Let $T$ be the diagonal subgroup of $\gl_n(\co)$. The space of weights $\Hom(T,\co^\ti)$ is free abelian spanned by the weights $L_1,\ld,L_n$ where $L_i$ is projection to the $i$-th entry. 

\subsection{Isotropy representation for $\Sl_n(\re)$}\label{sec:isotropySLn} Upon complexification, we need to study the isotropy representation of $\SO_n(\co)\sbs\Sl_n(\co)$. The adjoint representation of $\mf{sl}_n(\co)$ is isomorphic to the kernel of the contraction $V_n\ot V_n^*\ra\co$. As a $\mf{so}_n(\co)$-representation,
\[V_n\ot V_n^*\simeq\Lambda^2(V_n)+\sym^2(V_n).\]
$\Lambda^2(V_n)$ is the adjoint representation of $\mf{so}_n(\co)$. The representation $\sym^2(V_n)$ is not irreducible because $\mf{so}_n(\co)$ preserves a symmetric bilinear form $B:\sym^2(V_n)\ra\co$. The kernel of $B$ is the isotropy representation $\mf p$. 

There is a standard form $B$ for which the diagonal subgroup $H\sbs\SO(B)\simeq\SO_n(\co)$ is a Cartan subgroup (see \cite{fh}, pg.\ 268). For this choice, we have a standard basis $\mf h^*=\lan L_1,\ld,L_k\ran$, where $k=[n/2]$. If $n$ is even, the weights of $\mf p$ are $\pm L_i\pm L_j$ for $1\le i,j\le k$. If $n$ is odd, we have the additional weights $\pm L_i$ for $1\le i\le k$. 

As elements of $H^2(BH)\simeq\q[y_1,\ld,y_k]$, we have the following the total Chern classes for the isotropy representation. If $n$ is even, then
\begin{equation}\label{eqn:chernSL}c(\iota_\co)=\prod_{i,j}(1+y_i+y_j)(1+y_i-y_j)(1-y_i+y_j)(1-y_i-y_j)\end{equation}
and if $n$ is odd, then
\begin{equation}\label{eqn:chernSL2}c(\iota_\co)=\prod_i(1-y_i)(1+y_i)\prod_{i,j}(1+y_i+y_j)(1+y_i-y_j)(1-y_i+y_j)(1-y_i-y_j).\end{equation}

\subsection{Isotropy representation for $\SU_{p,q}$}\label{sec:isotropySUpq} 
Let $n=p+q$. Upon complexification we need to study the isotropy representation of $K_\co\sbs\Sl_n(\co)$, where $K_\co$ is the block diagonal subgroup 
\[K_\co=(\gl_p(\co)\ti\gl_q(\co))\cap\Sl_{n}(\co).\]
As described in Section \ref{sec:isotropySLn}, the adjoint representation of $\mf{sl}_n(\co)$ is a subspace of $V_n\ot V_n^*$. As a $\mf k_\co$-representation 
\[V_n\ot V_n^*=\big(V_p\ot V_p^*+V_q\ot V_q^*\big)+\big(V_1\ot V_2^*+V_1^*\ot V_2\big).\]
The adjoint representation of $\mf k_\co$ is a codimension-1 subspace of the first summand. The second summand is the isotropy representation $\mf p$. 

The diagonal subgroup $H\sbs K_\co$ coincides with the diagonal subgroup of $\Sl_n(\co)$, so we identify $\mf h^*$ as the quotient of $\lan L_1,\ld,L_n\ran$ by the subspace generated by $L_1+\cd+L_n$. The weights of the isotropy representation are $\pm(L_i-L_j)$ for $1\le i\le p$ and $p+1\le j\le p+q$. 

As an element of $H^2(BH)\sbs\q[y_1\ld,y_p,z_1,\ld,z_q]$, the total Chern class for the complexified isotropy representation is 
\begin{equation}\label{eqn:chernSU}c(\iota_\co)=\prod_{\substack{1\le i\le p\\1\le j\le q}}\big(1+(y_i-z_j)\big)\big(1-(y_i-z_j)\big)=\prod_{\substack{1\le i\le p\\1\le j\le q}}\big(1-(y_i-z_j)^2\big).\end{equation}

\subsection{Isotropy representation for $\SP_{2n}(\re)$}\label{sec:isotropySP2n} Upon complexification we need to study the isotropy representation of $\gl_n(\co)\sbs\SP_{2n}(\co)$, where $\gl_n(\co)$ is the subgroup of matrices of the form $\left(\begin{array}{cc}A&\\&(A^t)^{-1}\end{array}\right)$ for $A\in\gl_n(\co)$. 

As a $\mf{gl}_n(\co)$-representation, the standard representation of $\mf{sp}_{2n}(\co)$ decomposes $V_{2n}= V_n+V_n^*$. The adjoint representation of $\mf{sp}_{2n}(\co)$ is isomorphic to $\sym^2(V)$, and as a $\mf{gl}_n(\co)$-representation 
\[\sym^2(V)\simeq \big(V_n\ot V_n^*\big)+\big(\sym ^2(V_n)+\sym^2(V_n^*)\big).\]
$V_n\ot V_n^*$ is the adjoint representation of $\mf{gl}_n(\co)$ and $\sym ^2(V_n)+\sym^2(V_n^*)$ is the isotropy representation. 

Let $H\sbs\gl_n(\co)$ be the diagonal subgroup and identify $\mf h^*=\lan L_1,\ld,L_n\ran$ using the standard basis. The weights of the isotropy representation are $\pm (L_i+L_j)$ for $1\le i,j\le n$. As an element of $H^2(BH)\simeq\q[y_1,\ld,y_n]$, the total Chern class of the complexified isotropy representation is 
\begin{equation}\label{eqn:chernSP}
\begin{gathered}
\begin{array}{lll}
c(\iota_\co)&=&\prod_{i<j}(1+y_i+y_j)(1-(y_i+y_j))\prod_{i}(1+2y_i)(1-2y_i)\\[2mm]
&=&\prod_{i<j}\big(1-(y_i+y_j)^2\big)\prod_i\big(1-4y_i^2\big).
\end{array}\end{gathered}\end{equation}

\subsection{Isotropy representation for $\SO_{p,q}$}\label{sec:isotropySOpq} Let $n=p+q$ and $a=[p/2]$ and $b=[q/2]$. Upon complexification, we need to study the isotropy representation of $\SO_p(\co)\ti\SO_q(\co)\sbs\SO_n(\co)$. 

The adjoint representation of $\mf{so}_n(\co)$ is $\Lambda^2(V_n)$, and as a $\mf{so}_p(\co)\ti\mf{so}_q(\co)$-representation,
\[\Lambda^2(V_n)\simeq\big(\Lambda^2(V_p)+\Lambda^2(V_q)\big)+V_p\ot V_q.\]
$\Lambda^2(V_p)+\Lambda^2(V_q)$ is the adjoint representation of $\mf{so}_p(\co)\ti\mf{so}_q(\co)$, and $V_p\ot V_q$ is the isotropy representation. 

Let $H\sbs\SO_p(\co)\ti\SO_q(\co)$ be the standard Cartan subgroup, and identify $\mf h^*=\lan L_1,\ld,L_a, L_{a+1},\ld,L_{a+b}\ran$. In this basis, the weights of the isotropy representation are 
\[\left\{
\begin{array}{clll}
p,q\text{ even}&\pm L_i\pm L_j \\[1mm]
p\text{ even}, q\text{ odd}&\pm L_i\pm L_j,\>\>\> \pm L_i \\[1mm]
p\text{ odd}, q\text{ even}&\pm L_i\pm L_j,\>\>\> \pm L_j\\[1mm]
p\text{ odd}, q\text{ odd}&\pm L_i\pm L_j,\>\>\> \pm L_i,\>\>\>\pm L_j
\end{array}\right.\]
where $1\le i\le a$ and $a+1\le j\le a+b$.

As an element of $H^*(BH)\simeq\q[y_1,\ld,y_a,z_1,\ld,z_b]$, we have the following total Chern classes. If $p$ and $q$ are both even, then
\begin{equation}\label{eqn:chernSO1}
\begin{gathered}
\begin{array}{lll}
c(\iota_\co)&=&\prod (1+y_i+z_j)(1+y_i-z_j)(1-y_i+z_j)(1-y_i-z_j)\\[2mm]
&=&\prod_{\substack{1\le i\le a\\1\le j\le b}} \big(1-(y_i+z_j)^2\big)\big(1-(y_i-z_j)^2\big).
\end{array}\end{gathered}\end{equation}
If $p$ is even and $q$ is odd, then
\begin{equation}\label{eqn:chernSO2}c(\iota_\co)=\prod_{i=1}^a(1-y_i^2)\prod_{\substack{1\le i\le a\\1\le j\le b}} \big(1-(y_i+z_j)^2\big)\big(1-(y_i-z_j)^2\big).\end{equation}
If $p$ is odd and $q$ is even, then
\begin{equation}\label{eqn:chernSO3}c(\iota_\co)=\prod_{i=1}^b(1-z_j^2)\prod_{\substack{1\le i\le a\\1\le j\le b}} \big(1-(y_i+z_j)^2\big)\big(1-(y_i-z_j)^2\big).\end{equation}
If $p$ and $q$ are both odd, then
\begin{equation}\label{eqn:chernSO4}c(\iota_\co)=\prod_{i=1}^a(1-y_i^2)\prod_{j=1}^b(1-z_j^2)\prod_{\substack{1\le i\le a\\1\le j\le b}} \big(1-(y_i+z_j)^2\big)\big(1-(y_i-z_j)^2\big).\end{equation}

\subsection{Isotropy representation for $\SP_{p,q}$}\label{sec:isotropySPpq} Let $n=p+q$. After complexifying, we need to study the isotropy representation of $\SP_{2p}(\co)\ti\SP_{2q}(\co)\sbs\SP_{2n}(\co)$. 

The adjoint representation of $\mf{sp}_{2n}(\co)$ is $\sym^2(V_{2n})$, and as a $\mf{sp}_{2p}(\co)\ti\mf{sp}_{2q}(\co)$-representation, 
\[\sym^2(V_{2n})\simeq \big(\sym^2(V_{2p})+\sym^2(V_{2q})\big)+V_{2p}\ot V_{2q}.\]
$\sym^2(V_{2p})+\sym^2(V_{2q})$ is the adjoint representation of $\mf{sp}_{2p}(\co)\ti\mf{sp}_{2q}(\co)$, and $V_{2p}\ot V_{2q}$ is the isotropy representation. 

Let $H\sbs\SP_{2p}(\co)\ti\SP_{2q}(\co)$ be the standard Cartan subgroup, and identify 
$\mf h^*=\lan L_1,\ld,L_p,L_{p+1},\ld,L_{p+q}\ran$. The weights of the isotropy representation are $\pm L_i\pm L_j$, where $1\le i\le p$ and $2p+1\le j\le  2p+q$.
As an element of $H^2(BH)\simeq\q[y_1,\ld,y_p,z_1,\ld,z_q]$, the total Chern class of the complexified isotropy representation is 
\begin{equation}\label{eqn:chernSPpq}c(\iota_\co)=\prod_{\substack{1\le i\le p\\1\le j\le q}}(1+y_i+z_j)(1+y_i-z_j)(1-y_i+z_j)(1-y_i-z_j).\end{equation}

\subsection{Isotropy representation for $\SO^*_{2n}$}\label{sec:isotropySO*} Define $\SO^*_{2n}$ as the subgroup of $\gl_{2n}(\co)$ that preserves the Hermitian form and bilinear form defined by 
\[I_{n,n}=\left(\begin{array}{cc}I_n&\\&-I_n\end{array}\right)\>\>\>\>\>\text{ and }\>\>\>\>\>B_n=\left(\begin{array}{cc}0&I_n\\I_n&0\end{array}\right),\]respectively. We complexify and study the isotropy representation of $\gl_n(\co)\sbs\SO_{2n}(\co)$, where $\gl_n(\co)$ is the subgroup of matrices of the form $\left(\begin{array}{cc}A&\\&(A^t)^{-1}\end{array}\right)$ for $A\in\gl_n(\co)$. 

As a $\mf{gl}_n(\co)$-representation, the standard representation of $\mf{so}_{2n}(\co)$ decomposes $V_{2n}=V_n+V_n^*$. The adjoint representation of $\mf{so}_{2n}(\co)$ is $\Lambda^2(V_{2n})$, and as a $\mf{gl}_n(\co)$-representation, 
\[\Lambda^2(V_{2n})\simeq V_n\ot V_n^*+\big(\Lambda^2(V_n)+\Lambda^2(V_n^*)\big).\]
$V_n\ot V_n^*$ is the adjoint representation of $\mf{gl}_n(\co)$, and $\Lambda^2(V_n)+\Lambda^2(V_n^*)$ is the isotropy representation. 

Let $H\sbs\gl_n(\co)$ be the diagonal subgroup and identify $\mf h^*=\lan L_1,\ld,L_n\ran$. The weights of the isotropy representation are $\pm (L_i+L_j)$ for $1\le i<j\le n$. As an element of $H^2(BH)\simeq\q[y_1,\ld,y_n]$, the total Chern class is 
\begin{equation}\label{eqn:chernSO*}c(\iota_\co)=\prod_{i<j}\big(1+(y_i+y_j)\big)\big(1-(y_i+y_j)\big)=\prod_{i<j}\big(1-(y_i+y_j)^2\big).\end{equation}

\subsection{Isotropy representation for $\SU^*_{2n}$}\label{sec:isotropySU*} After complexifying, we need to study the isotropy representation of $\SP_{2n}(\co)\sbs\Sl_{2n}(\co)$. 

As described in Section \ref{sec:isotropySLn}, the adjoint representation of $\mf{sl}_{2n}(\co)$ is contained in $V_{2n}\ot V_{2n}^*$. As a $\mf{sp}_{2n}(\co)$-representation
\[V_{2n}\ot V_{2n}^*\simeq \sym^2(V_{2n})+\Lambda^2(V_{2n}).\]
$\sym^2(V_{2n})$ is the adjoint representation of $\mf{sp}_{2n}(\co)$. The representation $\Lambda^2(V_{2n})$ is not irreducible because by definition $\mf{sp}_{2n}(\co)$ preserves an antisymmetric bilinear form $J:\Lambda^2(V_{2n})\ra\co$. The kernel of $J$ is the isotropy representation $\mf p$. 

Let $H\sbs\SP_{2n}(\co)$ be the standard Cartan subgroup, and identify $\mf h^*=\lan L_1,\ld,L_n\ran$. The nonzero weights of the isotropy representation are $\pm L_i\pm L_j$ for $1\le i<j\le n$. As an element of $H^2(BH)$, the total Chern class of the isotropy representation is
\begin{equation}\label{eqn:chernSU*}c(\iota_\co)=\prod_{1\le i<j\le n}(1+y_i+y_j)(1+y_i-y_j)(1-y_i+y_j)(1-y_i-y_j).\end{equation}

\subsection{Isotropy representation for real forms of $E_8$}\label{sec:isotropyE8} 
Let $\mf h^*(\mf e_8)$ be the dual to the Cartan subalgebra of $\mf e_8$. As in \cite{adams} (pg.\ 56), we identify $\mf h_8^*=\lan J_1,\ld,J_8\ran$. 

In the remainder of this section, we use $\la_n^i$ to denote the $i$-th exterior power of the standard representations of the real Lie algebras $\mf{su}_n$, $\mf{so}_n$, and $\mf{sp}_n$. 

\vspace{.1in}

\noindent\un{$E_{8(8)}$.} The maximal compact subgroup $K\sbs E_{8(8)}$ has Lie algebra $\mf{so}_{16}$. As an $K$-representation, $\mf e_8$ decomposes 
\[\mf e_8=\mf{so}_{16}\op\De^+\]
where $\De^+$ is the positive spin representation of $\mf{so}_{16}$ (see \cite{adams}, chapter 6). Then the isotropy representation of $K\sbs E_8$ is the spin representation. Let $\mf h^*(\mf{so}_{16})$ denote the dual to the Cartan subalgebra of $\mf{so}_{16}$. We identify $\mf h^*(\mf{so}_{16})=\lan L_1,\ld, L_8\ran$. The weights of the isotropy representation are
\[\fr{1}{2}(\pm L_1\pm L_2\pm\cd\pm L_8)\]
where the number of + signs is even. 

Let $S\sbs K$ be a maximal torus. As an element of $H^*(BS)\simeq \q[y_1,\ld,y_8]$, the total Chern class of the complexified isotropy representation is 
\begin{equation}\label{eqn:chernE88}c(\iota_\co)=\prod\left(1-\fr{1}{2}(\ep_1y_1+\cd+\ep_8y_8)\right),\end{equation}
where the product is over all tuples $(\ep_1,\ld,\ep_8)\in\{\pm1\}^8$ that have an even number of $-$'s. 

\vspace{.1in}

\noindent\un{$E_{8(-24)}$.} The maximal compact subgroup $K\sbs E_{8(-24)}$ has Lie algebra $\mf e_7\ti\mf{su}_2$. As a $K$-representation, $\mf e_8$ decomposes
\[\mf e_8=(\mf e_7\>\op\>\mf{su}_2)\>\op\> (V\ot\la^1_2),\]
where $V$ is the 56-dimensional representation of $E_7$. See \cite{adams}, pg.\ 54.

Since $K\sbs E_8$ have the same rank, we can identify the (dual) Cartan subalgebras $\mf h^*(\mf e_8)=\lan J_1,\ld,J_8\ran$ and $\mf h^*(K)=\lan J_1,\ld,J_6,\fr{J_7+J_8}{2}\ran\op\lan\fr{J_7-J_8}{2}\ran$.

The roots for $\mf e_7\ti\mf{su}_2$ are $\pm(J_7-J_8)$ together with all the roots of $\mf e_8$ that are orthogonal to $J_7-J_8$. The roots of $\mf e_8$ that are not roots of $\mf e_7\ti\mf{su}_2$ correspond to the weights of the isotropy representation $\mf p=V\ot\la^1_2$. Then the weights of $\mf p$ are 
\[\left\{\begin{array}{cccl}
\pm J_i\pm (\fr{J_7+J_8}{2})\pm(\fr{J_7-J_8}{2})&1\le i\le 6\\[2mm]
\fr{\pm J_1\pm\cd\pm J_6}{2}\pm\fr{J_7-J_8}{2}
&\text{even number of $J_i$ with $-$ sign}
\end{array}\right.\]

Let $S\sbs K$ be a maximal torus. Define $y_i=\tau(J_i)\in H^2(BS)$ for $1\le i\le 6$ and $y_7=\tau(J_7+J_8)$ and $y_8=\tau(J_7-J_8)$. The total Chern class for the complexified isotropy representation $c(\iota_\co)\in H^*(BS)^W$ is 
\begin{equation}\label{eqn:chernE824}
c(\iota_\co)=\prod\left(1-\big(\ep y_i+\ep_7\cdot\fr{y_7}{2}+\ep_8\cdot\fr{y_8}{2}\big)\right)\prod\left(1-\fr{1}{2}\big(\ep_1y_1+\cd+\ep_6y_6+\ep_8y_8\big)\right)
\end{equation}
The first product is over $i=1,\ld,6$ and all tuples $(\ep,\ep_7,\ep_8)\in\{\pm1\}^3$. The second product is over the tuples $(\ep_1,\ld,\ep_6,\ep_8)\in\{\pm1\}^7$ with an even number of $-1$'s. 

\subsection{Isotropy representation for real forms of $E_7$}\label{sec:isotropyE7}\mbox{ }
%Let $\mf{su}_2\sbs \mf e_8$ be the subalgebra whose roots are $\pm(J_7-J_8)$. The centralizer of $\mf{su}_2$ is isomorphic to $\mf e_7$ and its roots are the following subset of the roots of $\mf e_8$.\[\left\{\begin{array}{cc}\pm(J_7-J_8)\\[1mm]\pm J_i\pm J_j& 1\le i<j\le 6\\[1mm]\fr{1}{2}\big(\pm J_1\pm J_2\pm\cd\pm J_6\pm(J_7-J_8)\big)& \text{ the number of $J_i$ with a + sign is even}\end{array}\right.\]

\vspace{.1in} 

\noindent\un{$E_{7(7)}$.} The maximal compact subgroup $K\sbs E_{7(7)}$ has Lie algebra $\mf{su}_8$. Following \cite{adams} (pg.\ 69), as a $K$-representation, $\mf e_7$ decomposes
\[\mf e_7=\mf{su}_8\op\lambda^4_8.\]
Here $\mf p=\lambda^4_8$ is the isotropy representation, and the weights are $L_{i_1}+L_{i_2}+L_{i_3}+L_{i_4}$ for $1\le i_1<\cd<i_4\le 8$. 

Let $S\sbs K$ be a maximal torus. We identify $H^*(BS)$ with the quotient of $\q[y_1,\ld,y_8]$ by the ideal generated by $y_1+\cd+y_8$. The total Chern class of the complexified isotropy representation is 
\begin{equation}\label{eqn:chernE77}c(\iota_\co)=\prod_{1\le i_1<i_2<i_3<i_4\le 8}\big(1-(y_{i_1}+y_{i_2}+y_{i_3}+y_{i_4})\big).\end{equation}

\vspace{.1in}

\noindent\un{$E_{7(-5)}$.} The maximal compact subgroup $K\sbs E_{7(-5)}$ has Lie algebra $\mf{so}_{12}\ti\mf{su}_2$. Following \cite{adams} (pg.\ 52), $\mf e_7$ decomposes as a $K$-representation as
\[\mf e_7=(\mf{so}_{12}\op\mf{su}_2)\>\op\>\De^+\ot\lambda^1_2\]
where $\De^+$ is the positive spin representation of $K$. Here $\mf p=\De^+\ot\lambda^1$ is the isotropy representation. Identify $\mf h^*(K)\simeq\lan L_1,\ld,L_6\ran\op\lan L_7\ran$ in the standard way. The weights of the isotropy representation are $\fr{1}{2}\big(\ep_1 L_1+\cd+\ep_6 L_6)\pm L_7$, where $(\ep_1,\ld,\ep_6)\in\{\pm1\}^6$ has an even number of $-1$'s. 

Let $S\sbs K$ be a maximal torus. As an element of $H^*(BS)\simeq\q[y_1,\ld,y_6]\ot\q[y_7]$, the total Chern class for the complexified isotropy representation is 
\begin{equation}\label{eqn:chernE75}c(\iota_\co)=\prod\left(1-\fr{1}{4}(\ep_1y_1+\cd+\ep_6y_6+y_7)^2\right),\end{equation}
where $(\ep_1,\ld,\ep_6)\in\{\pm1\}^6$ has an even number of $-1$'s. 

\vspace{.1in} 

\noindent\un{$E_{7(-25)}$.} The maximal compact subgroup $K\sbs E_{7(-25)}$ has Lie algebra $\mf e_6\ti\mf{so}_2$. In a similar way to \cite{adams} (pg.\ 52), one shows that $\mf e_7$ decomposes as 
 a $K$-representation:
\[\mf e_7=(\mf e_6\op\mf{so}_2)\>\op \>U_1\ot\xi^2\op U_2\ot\xi^{-2}\]
where $U_1$ and $U_2$ are the 27-dimensional representations of $E_6$. Identify  
\[\mf h^*(\mf e_6)=\lan J_1,\ld,J_5, J_6+J_7+J_8\ran\>\>\>\>\>\text{ and }\>\>\>\>\>\mf h^*(\mf e_7)=\lan J_1,\ld,J_6, J_7+J_8\ran\]
as subspaces of $h^*(\mf e_8)$ (see Section \ref{sec:background}). 
The orthogonal complement of $\mf h^*(\mf e_6)$ in $\mf h^*(\mf e_7)$ is 1-dimensional and generated by $-2J_6+J_7+J_8$. We have
\[\mf h^*(\mf e_6\ti\mf{so}_2)=\lan J_1,\ld,J_5, J_6+J_7+J_8\ran\op\lan-2J_6+J_7+J_8\ran.\]
The weights of the isotropy representation are the roots of $\mf e_7$ that are not roots of $\mf e_6\ti\mf{so}_2$. These are $\pm (J_7+J_8)$, $\pm J_i\pm J_6$ for $1\le i\le 5$, and 
\[\fr{1}{2}\big(\ep_1J_1+\cd+\ep_5J_5+\ep_6(J_6-J_7-J_8)\big)\]
where $(\ep_1,\ld,\ep_6)\in\{\pm1\}^6$ has even number of $-1$'s. To write these roots as weights of $\mf e_6\ti\mf{so}_2$, note that 
\[J_7+J_8=\fr{1}{3}\big[2(J_6+J_7+J_8)+(-2J_6+J_7+J_8)\big]\]
and
\[J_6=\fr{1}{3}\big[(J_6+J_7+J_8)-(-2J_6+J_7+J_8)\big].\]
The weights of the isotropy representation are 
\begin{itemize}
\item \[\pm\fr{1}{3}\big[2(J_6+J_7+J_8)+(-2J_6+J_7+J_8)\big],\]
\item for $1\le i\le 5$, \[\pm J_i\pm\fr{1}{3}\big[(J_6+J_7+J_8)-(-2J_6+J_7+J_8)\big]\] 
\item for $(\ep_1,\ld,\ep_6)\in\{\pm1\}^6$ with an even number of $-1$'s, 
\[\fr{1}{2}\left(\ep_1J_1+\cd+\ep_5J_5+\ep_6\cdot\fr{1}{3}\big[-(J_6+J_7+J_8)-2(-2J_6+J_7+J_8)\big]\right)\]
\end{itemize}
%\[\begin{array}{cccc}
%\pm\fr{1}{3}\big[2(J_6+J_7+J_8)+(-2J_6+J_7+J_8)\big] &\\[2mm]
%\pm J_i\pm\fr{1}{3}\big[(J_6+J_7+J_8)-(-2J_6+J_7+J_8)\big]& 1\le i\le 5\\[2mm]
%\fr{1}{2}\left(\ep_1J_1+\cd+\ep_5J_5+\ep_6\cdot\fr{1}{3}\big[-(J_6+J_7+J_8)-2(-2J_6+J_7+J_8)\big]\right)&(\ep_1,\ld,\ep_6)\text{ even $-1$'s}
%\end{array}\]

Let $S\sbs E_6\ti\SO_2$ be a maximal torus. Let $z_i=\tau(J_i)\in H^2(BS)$ for $1\le i\le 5$, and let $z_6=\tau(J_6+J_7+J_8)$ and let $z_7=\tau(-2J_6+J_7+J_8)$. As elements of $H^2(BS)$, the weights of the isotropy representation are 
\begin{itemize}
\item $\pm(2z_6+z_7)$, 
\item $\pm z_i\pm\fr{1}{3}(z_6-z_7)$ for $1\le i\le 5$, and 
\item $\fr{1}{2}\big(\ep_1z_1+\cd+\ep_5z_5+\ep_6\cdot\fr{1}{3}(-z_6-2z_7)\big)$, where $(\ep_1,\ld,\ep_6)\in\{\pm1\}^6$ has even number of $-1$'s. 
\end{itemize}
The total Chern class is 
\begin{equation}\label{eqn:chernE725}\begin{gathered}
\begin{array}{lll}c(\iota_\co)&=&\left(1-\fr{1}{9}(2z_6+z_7)^2\right)\prod\left(1-(z_i+\fr{z_6-z_7}{3})^2\right)\left(1-(z_i-\fr{z_6-z_7}{3})^2\right)\\[2mm]
&&\prod\left(1-\fr{1}{4}\left(\ep_1z_1+\cd+\ep_5z_5+\ep_6\fr{-z_6-2z_7}{3}\right)^2\right).\end{array}\end{gathered}\end{equation}
The first product is over $1\le i\le 5$. The second product is over $(\ep_1,\ld,\ep_6)\in\{\pm1\}^6$ with an even number of $-1$'s. 

\subsection{Isotropy representation for real forms of $E_6$}\label{sec:isotropyE6} \mbox{ }
%Let $\mf{su}_3\sbs\mf e_8$ be the subalgebra whose roots are $\pm(J_6-J_7)$, $\pm(J_6-J_8)$, and $\pm(J_7-J_8)$. The centralizer of $\mf{su}_3$ is isomorphic to $\mf e_6$ and its roots are \[\left\{\begin{array}{cc}\pm J_i\pm J_j& 1\le i<j\le 5\\[1mm]\fr{1}{2}\big(\pm J_1\pm\cd\pm J_5\pm(J_6+J_7+J_8)\big)& \text{ the number of $J_i$ with a $-$ sign is even}\end{array}\right.\]

\vspace{.1in}

\noindent\un{$E_{6(6)}$} The maximal compact subgroup $K\sbs E_{6(6)}$ has Lie algebra $\mf{sp}_4$. As a $K$-representation, $\mf e_6$ decomposes as
\[\mf e_6=\mf{sp}_4\op W\]
where $W\sbs\la^4_8$ is the kernel of the contraction map $\la^4_8\ra\la^2_8$. It is the irreducible representation of $K$ with highest weight $L_1+L_2+L_3+L_4$. So $\mf p=W$ and the nonzero weights of the isotropy representation are 
\[\begin{array}{cll}
\pm L_i\pm L_j&1\le i<j\le 4\\[1mm]
\ep_1L_1+\cd+\ep_4L_4&(\ep_1,\ld,\ep_4)\in\{\pm1\}^4.
\end{array}\]
Let $S\sbs\SP_4$ be a maximal torus. As an element of $H^*(BS)\simeq\q[y_1,\ld,y_4]$, the total Chern class is 
\begin{equation}\label{eqn:chernE66}c(\iota_\co)=\prod_{1\le i<j\le 4}\big(1-(y_i+y_j)^2\big)\big(1-(y_i-y_j)^2\big)\prod\big(1-(\ep_1y_1+\ep_2y_2+\ep_3y_3+y_4)^2\big).\end{equation}
The second product is over all tuples $(\ep_1,\ep_2,\ep_3)\in\{\pm1\}^3$. 

\vspace{.1in}

\noindent\un{$E_{6(2)}$} The maximal compact subgroup $K\sbs E_{6(2)}$ has Lie algebra $\mf{su}_6\ti\mf{su}_2$. As a $K$-representation, $\mf e_6$ decomposes as
\[\mf e_6=(\mf{su}_6\op\mf{su}_2)\>\op\> (\la^3_6\ot\la^1_2).\]
With respect to the standard basis for $\mf h^*(\mf{su}_6\ti\mf{su}_2)$, the isotropy representation $\mf p=\la^3_6\ot\la^1_2$ has weights $L_{i_1}+L_{i_2}+L_{i_3}\pm L_7$, where $1\le i_1<i_2<i_3\le 6$. 

Let $S\sbs K$ be a maximal torus. As an element of $\q[y_1,\ld,y_7]\onto H^*(BS)$ the total Chern class is 
\begin{equation}\label{eqn:chernE62}c(\iota_\co)=\prod_{1\le i_1<i_2<i_3\le 6}\big(1-(y_{i_1}+y_{i_2}+y_{i_3}+y_7)^2\big).\end{equation}

\vspace{.1in}

\noindent\un{$E_{6(-14)}$} The maximal compact subgroup $K\sbs E_{6(-14)}$ has Lie algebra $\mf{so}_{10}\ti\mf{so}_2$. According to \cite{adams} (pg.\ 53), as a $K$-representation $\mf e_6$ decomposes as
\[\mf e_6=\big(\mf{so}_{10}+\mf{so}_2\big)\oplus\big(\De^+\ot\xi^3+\De^-\ot\xi^{-3}\big)\]
where $\xi:\SO_2\ra\co^\ti$ denotes the identity representation of $\SO_2$ (or rather the induced Lie algebra representation), and $\xi^k:\SO_2\ra\co^\ti$ denotes the $k$-th power of $\xi$. 

With respect to the standard basis $\mf h^*(\mf{so}_{10}\ti\mf{so}_2)=\lan L_1,\ld,L_5\ran\op\lan L_6\ran$, the isotropy representation $\mf p=\De^+\ot\xi^3+\De^-\ot\xi^{-3}$ has weights 
\[\fr{1}{2}\big(\ep_1L_1+\cd+\ep_5L_5)+3\ep_6L_6,\]
where $(\ep_1,\cd,\ep_6)\in\{\pm1\}^6$ has an even number of $-1$'s.

Let $S\sbs K$ be a maximal torus. As an element of $H^*(BS)\simeq\q[y_1,\ld,y_5]\ot\q[y_6]$, the total Chern class of the isotropy representation is 
\begin{equation}\label{eqn:chernE614}c(\iota_\co)=\prod\left(1-\fr{1}{4}(\ep_1y_1+\cd+\ep_5y_5+6y_6)^2\right),\end{equation}
with the product over the tuples $(\ep_1,\ld,\ep_5)\in\{\pm1\}^5$ with an even number of $-1$'s. 

\vspace{.1in}

\noindent\un{$E_{6(-26)}$} The maximal compact subgroup $K\sbs E_{6(-26)}$ has Lie algebra $\mf f_4$. Following \cite{adams} (pg.\ 95), as a $K$-representation, $\mf e_6$ decomposes
\[\mf e_6=\mf f_4\>\op\>U\]
where $U$ is the 26-dimensional fundamental representation of $\mf f_4$. With respect to the standard basis $\mf h^*(\mf f_4)=\lan L_1,\ld,L_4\ran$, the isotropy representation $\mf p\simeq U$ has weights $\pm L_i$ for $1\le i\le 4$ and $\fr{1}{2}(\pm L_1\pm L_2\pm L_3\pm L_4)$.

Let $S\sbs F_4$ be a maximal torus. As an element of $H^*(BS)\simeq\q[y_1,y_2,y_3,y_4]$, the total Chern class of the isotropy representation is 
\begin{equation}\label{eqn:chernE626}c(\iota_\co)=\prod_{i=1}^4(1-y_i^2)\prod_{(\ep_1,\ep_2,\ep_3)\in\{\pm1\}^3}\left(1-\fr{1}{4}(\ep_1y_1+\ep_2y_2+\ep_3y_3+y_4)^2\right).
\end{equation}

\subsection{Isotropy representation for real forms of $F_4$}\label{sec:isotropyF4}\mbox{ }

\vspace{.1in}

\noindent \un{$F_{4(4)}$} The maximal compact subgroup $K\sbs F_{4(4)}$ has Lie algebra $\mf{su}_2\ti\mf{sp}_3$. As a $K$-representation, $\mf f_4$ decomposes as 
\[\mf f_4=(\mf{su}_2\op\mf{sp}_3)\>\op\>\la^1_2\ot W\]
where $W\sbs\la^3_6$ is the kernel of the contraction map $\la^3_6\ra\la^1_6$ (it is an irreducible representation of $\SP_3$ of dimension 14). 

Since $K\sbs F_{4(4)}$ have the same rank, we can identify $\mf h^*(\mf f_4)=\lan J_1,J_2,J_3,J_4\ran$ and $\mf h^*(\mf{su}_2\ti\mf{sp}_3)=\lan L_1\ran\op\lan L_2,L_3, L_4\ran$. 

The roots of $F_4$ are
\[\left\{\begin{array}{ccl}
\pm J_i& 1\le i\le 4\\[1mm]
\pm J_i\pm J_j& 1\le i<j\le 4 \\[1mm]
\fr{1}{2}(\pm J_1\pm J_2\pm J_3\pm J_4)
\end{array}\right.\]
According to \cite{roth} (pg.\ 390), the roots of $\mf{su}_2\ti\mf{sp}_3\sbs\mf f_4$ are \[\begin{array}{ccl}\pm (J_1-J_2),\>\>\pm (J_1+J_2),\>\> \pm J_3,\>\>\pm J_4,\>\>\pm J_3\pm J_4, \>\>\\[1mm]\fr{1}{2}\big(\pm(J_1+J_2)\pm J_3\pm J_4\big)\\[1mm]\end{array}\]
We can identify these roots with the roots of $\mf{su}_2\ti\mf{sp}_3$ inside $\mf h^*(\mf{su}_2\ti\mf{sp}_3)$ by the identification 
\[\begin{array}{llll}
J_1-J_2&\leftrightarrow& 2L_1\\[1mm]
J_1+J_2&\leftrightarrow&2L_2\\[1mm]
J_3+J_4&\leftrightarrow& 2L_3\\[1mm]
J_3-J_4&\leftrightarrow& 2L_4
\end{array}\]
Under this identification, the weights of the isotropy representation $\mf p=W\ot\la^1_2$ will be the roots of $\mf f_4$ that are not roots of $\mf{su}_2\ti\mf{sp}_3$. Then the weights of the isotropy representation are $\pm L_1\pm L_i$ for $i=2,3,4$ and $\pm L_1\pm L_2\pm L_3\pm L_4$ (for all 16 sign choices).

Let $S\sbs K$ be a maximal torus. As an element of $H^*(BS)\simeq\q[y_1]\ot\q[y_2,y_3,y_4]$, the total Chern class for the isotropy representation is 
\begin{equation}\label{eqn:chernF44}
c(\iota_\co)=\prod_{i=2}^4\big(1-(y_1+y_i)^2\big)\big(1-(y_1-y_i)^2\big)\prod\big(1-(y_1+\ep_2y_2+\ep_3y_3+\ep_4y_4)^2\big)
\end{equation}
where the second product is over all tuples $(\ep_2,\ep_3,\ep_4)\in\{\pm1\}^3$. 

\vspace{.1in}

\noindent\un{$F_{4(-20)}$} The maximal compact subgroup $K\sbs F_{4(-20)}$ has Lie algebra $\mf{so}_9$. Following \cite{adams} (pg.\ 51), $\mf f_4$ decomposes as a representation of $K$ as
\[\mf f_4=\mf{so}_9\op\De\]
where $\De$ is the spin representation of $K$. The isotropy representation $\mf p=\De$ has weights $\fr{1}{2}(\pm L_1\pm\cd\pm L_4)$. 

Let $S\sbs K$ be the maximal torus. As an element of $H^*(BS)\simeq\q[y_1,\ld,y_4]$, the total Chern class of the isotropy representation is 
\begin{equation}\label{eqn:chernF420}c(\iota_\co)=\prod_{(\ep_1,\ep_2,\ep_3)\in\{\pm1\}^3}\left(1-\fr{1}{4}(\ep_1y_1+\ep_2y_2+\ep_3y_3+y_4)^2\right).
\end{equation}

\subsection{Isotropy representation for real forms of $G_2$}\label{sec:isotropyG2} The maximal compact subgroup of $G_{2(2)}$ is $K=\SU_2\ti\SU_2$. As a $K$-representation, $\mf g_2$ decomposes
\[\mf g_2=\big(\mf{su}_2\ti\mf{su}_2)+\sym^2(V)\ot V\]
where $V$ is the standard representation of $\SU_2$, and $\sym^2(V)$ is the 2nd symmetric power. The weights of $\mf p=\sym^2(V)\ot V$ are $\pm L_1\pm L_2$ and $\pm 3L_1\pm L_2$. (Compare with Lecture 22 of \cite{fh} or pg.\ 393 of \cite{roth}.)

Let $S\sbs \SU_2\ti\SU_2$ be a maximal torus. As an element of $H^*(BS)\simeq\q[y_1,y_2]$, the total Chern class is 
\begin{equation}\label{eqn:chernG22}c(\iota_\co)=(1-(y_2+3y_1)^2)(1-(y_2+y_1)^2)(1-(y_2-y_1)^2)(1-(y_2-3y_1)^2).\end{equation}

\section{Computing $\al_2^*:H^*(BG)\ra H^*(BG^\de)$}\label{sec:cw} 

The following theorem allows us to compute $\al_2: H^*(BG)\ra H^*(BG^\de)$ for $G$ a real, semisimple Lie group. See Milnor \cite{milnor_liediscrete}, Theorem 2.

\begin{thm}[Milnor]\label{thm:milnor}
Let $G$ be a real, simple, connected Lie group. Assume that its complexification $G_\co$ is simple. Then sequence
\[H^*(B G_\co;\q)\xra{i^*} H^*(BG;\q)\xra{\al_2^*} H^*(BG^\de;\q)\]
induced by the maps $G^\de\ra G\ra G_\co$ is ``exact" in the sense that the kernel of $\al_2^*$ is the ideal generated by the image of $i^*: H^k(BG_\co)\ra H^k(BG)$ for $k>0$.\end{thm}

Theorem \ref{thm:milnor} applies to all the groups $G$ in Table 2. If $G$ is one of the complex Lie groups from Table 1, then $G_\co\simeq G\ti G$ is not simple, so Theorem \ref{thm:milnor} does not apply. In this case, we have the following theorem, whose proof comes from Chern-Weil theory (see Milnor \cite{milnor_liediscrete} Lemma 11). 

\begin{thm}\label{thm:milnor2}
Let $G$ be a complex, simple Lie group with finitely many components. Then $\al_2^*:H^i(BG;\q)\ra H^i(BG^\de;\q)$ is zero for $i>0$.
\end{thm}

By Theorem \ref{thm:milnor2}, if $G$ is one of the groups from Table 1, then $\al_2^*$ is zero in positive degrees and so $p_i(\Ga\bs G/K)=0$ for $i>0$. This proves Theorem \ref{thm:main} for the groups in Table 1. 

In the remainder of this section we use Theorem \ref{thm:milnor} to compute $\al_2^*$ for all the $G$ in Table 2. Let $G$ be one of the groups in Table 2 with maximal compact subgroup $K$. Let $G_\co$ be the complexification of $G$ and let $U$ be the maximal compact subgroup of $G_\co$. Let $S\sbs K$ and $S'\sbs U$ be maximal tori such that complexification $G\ra G_\co$ sends $S$ into $S'$. We have the following diagram of inclusions, and an induced diagram on cohomology.

\[\begin{xy}
(-40,15)*+{G}="A";
(-20,15)*+{G_\co}="B";
(-40,0)*+{K}="C";
(-20,0)*+{U}="D";
(-40,-15)*+{S}="E";
(-20,-15)*+{S'}="F";
{\ar"A";"B"}?*!/_3mm/{};
{\ar "D";"B"}?*!/_3mm/{};
{\ar "C";"A"}?*!/^3mm/{};
{\ar "C";"D"}?*!/_3mm/{};
{\ar "E";"F"}?*!/^3mm/{};
{\ar "E";"C"}?*!/_3mm/{};
{\ar "F";"D"}?*!/_3mm/{};
(10,15)*+{H^*(BG)}="G";
(40,15)*+{H^*(BG_\co)}="H";
(10,0)*+{H^*(BK)}="I";
(40,0)*+{H^*(BU)}="J";
(10,-15)*+{H^*(BS)^W}="K";
(40,-15)*+{H^*(BS')^{W'}}="L";
(10,-30)*+{H^*(BS)}="M";
(40,-30)*+{H^*(BS')}="N";
{\ar"H";"G"}?*!/^3mm/{i^*};
{\ar "H";"J"}?*!/_3mm/{\simeq};
{\ar "G";"I"}?*!/_3mm/{\simeq};
{\ar "J";"I"}?*!/_3mm/{};
{\ar@{-->} "L";"K"}?*!/^3mm/{j^*};
{\ar "I";"K"}?*!/_3mm/{\simeq};
{\ar "J";"L"}?*!/_3mm/{\simeq};
{\ar "N";"M"}?*!/^3mm/{};
{\ar@{^{(}->} "K";"M"}?*!/_3mm/{};
{\ar@{^{(}->} "L";"N"}?*!/_3mm/{};
\end{xy}\]

\vspace{.1in}
In the diagram on the right, the top two vertical arrows are isomorphisms because $K\hra G$ and $U\hra G_\co$ are homotopy equivalences. The middle two vertical arrows are isomorphisms by Theorem \ref{thm:borel}. Thus to compute the map $i^*$ it is enough to compute $j^*$. We do this  in the remainder of this section. 

\subsection{Computing $\al_2^*$ for $G=\Sl_n(\re)$}\label{sec:cwSLn} Here $K=\SO_n$ and $U=\SU_n$. For $\ta\in\re/2\pi\z$, let 
\[\un{\ta}=\left(\begin{array}{rc}\cos(\ta)&\sin(\ta)\\-\sin(\ta)&\cos(\ta)\end{array}\right).\]
Let $k=[n/2]$. The map $S\ra S'$ sends
\begin{equation}\label{eqn:weightsSO(n)}\left(\begin{array}{ccccc}\un{\ta}_1&&\\&\ddots&\\&&\un{\ta}_k\\&&&(1)\end{array}\right)\mapsto\left(\begin{array}{llc}D&\\&D^{-1}\\&&(1)\end{array}\right)\end{equation}
where 
\[D=\left(\begin{array}{ccc}e^{i\ta_1}&&\\&\ddots&\\&&e^{i\ta_k}\end{array}\right),\]
and the 1 in (\ref{eqn:weightsSO(n)}) appears only if $n$ is odd. Identify 
\[H^*(BS')^{W'}\simeq\fr{\sym(x_1\ld,x_n)}{(x_1+\cd+x_n)}\]
and
\[H^*(BS)^W\simeq\left\{
\begin{array}{clll}
\sym(y_1^2,\ld,y_k^2)&n=2k+1\\[2mm]
\lan\sym(y_1^2,\ld,y_k^2), y_1\cd y_k\ran& n=2k
\end{array}\right.\]

\vspace{.05in}
\noindent From (\ref{eqn:weightsSO(n)}) we conclude that the image of $j^*:H^*(BS')^{W'}\ra H^*(BS)^W$ is $\sym(y_1^2,\ld,y_k^2)$.

\subsection{Computing $\al_2^*$ for $G=\SU_{p,q}$}\label{sec:cwSUpq} Here  $K=S(U_p\ti U_q)$ and $U=\SU_{p+q}$, and the inclusion $S\hra S'$ is an isomorphism. 
Identify 
\[H^*(BS')^{W'}\simeq\fr{\sym(x_1\ld,x_{p+q})}{(x_1+\cd+x_{p+q})}\]
and
\[H^*(BS)^{W}\simeq\fr{\sym(y_1\ld,y_p)\ot\sym(z_1,\ld,z_q)}{(y_1+\cd+y_p+z_1+\cd+z_{q})}.\]
The image of $j^*:H^*(BS')^{W'}\ra H^*(BS)^W$ is $\sym(y_1,\ld,y_p,z_1,\ld, z_q)$. 

\begin{cor}\label{cor:SUpq2}
Let $p,q\ge2$ and let $G=\emph{\SU}_{p,q}$. Then $\al_2^*:H^*(BG)\ra H^*(BG^\de)$ is injective on the linear subgroup generated by $\sum_{i<j}y_iy_j$ and $\big(\sum y_i\big)^2$.
\end{cor}
\begin{proof}
For $1\le i\le p$ (resp.\ $1\le i\le q$), let $\si_i^y$ (resp.\ $\si_i^z$) denote the $i$-th elementary symmetric polynomial in $\{y_1,\ld,y_p\}$ (resp.\ $\{z_1,\ld,z_q\}$), viewed as elements of $H^*(BS)^W\simeq H^*(BG)$. Similarly, let $\si_i^{y,z}$ denote the $i$-th elementary symmetric polynomial in $\{y_1,\ld,y_p,z_1,\ld,z_q\}$. Let $\ca I$ denote the ideal  \[\ca I=(\si_1^{y,z},\ld,\si_{p+q}^{y,z}).\] 
The following relations are easy to verify.
\begin{equation}
\label{eqn:corSUpq2}\si_1^{y,z}=\si_1^y+\si_1^z,\>\>\>\>\>\text{ and }\>\>\>\>\>\si_2^{y,z}=\si_2^y+\si_1^y\si_1^z+\si_2^z.\end{equation}
Note that $(\si_1^y)^2=\big(\sum y_i\big)^2$ and $\si_2^y=\sum_{i<j} y_iy_j$. By Theorem \ref{thm:milnor}, to prove the corollary it is enough to show that no nontrivial linear combination $c\cdot \si_2^y+d\cdot(\si_1^y)^2$ belongs to $\ca I$. The elements of $\ca I$ that have total degree 2 all have the form 
\[(a\si_1^y+b\si_1^z)\cdot \si_1^{y,z}+m\cdot \si_2^{y,z}\]
where $a,b,m$ are scalars. By simple linear algebra, one checks that if
\[c\cdot \si_2^y+d\cdot(\si_1^y)^2=(a\si_1^y+b\si_1^z)\cdot \si_1^{y,z}+m\cdot \si_2^{y,z},\] then $m=a=b=0$. This implies $c=d=0$ because $\si_2^y$ and $(\si_1^y)^2$ are linearly independent in $H^*(BS)^W$. Hence no nontrivial linear combination $c\cdot \si_2^y+d\cdot(\si_2^y)^2$ belongs to $\ca I$.
\end{proof}

\begin{cor}\label{cor:SUpq}
Let $p\ge2$ and let $G=\emph{\SU}_{p,1}$. Then $\al_2^*:H^*(BG)\ra H^*(BG^\de)$ is injective on $\sum_{i<j}y_iy_j$, and 
\[\al_2^*\big(\sum y_i\big)^2=\al_2^*\big(\sum_{i<j}y_iy_j\big).\]
\end{cor}

\begin{proof}
We use the same notation as in the proof of Corollary \ref{cor:SUpq2}. Since $q=1$ in the present case, the relations in (\ref{eqn:corSUpq2}) simplify to the following relations.
\begin{equation}
\label{eqn:corSUpq}\si_1^{y,z}=\si_1^y+\si_1^z,\>\>\>\>\>\text{ and }\>\>\>\>\>\si_2^{y,z}=\si_2^y+\si_1^y\si_1^z.\end{equation}
Since the elements $\si_i^y, \si_j^z$ are all linearly independent in $H^2(BS)^W$, it is a simple matter of linear algebra to show that $\si_2^y$ is not in $\ca I$, and hence is not in the kernel of $\al_2^*$. 

To see that $\al_2^*\big((\si_1^y)^2\big)=\al_2^*(\si_2^y)$, note that the first equation in (\ref{eqn:corSUpq}) implies that 
\[\si_1^y\cdot\si_1^{y,z}=(\si_1^y)^2+\si_1^y\si_1^z\in\ca I.\] 
Then $\si_2^y-(\si_1^y)^2=\si_2^{y,z}-\si_1^y\cdot\si_1^{y,z}$ belongs to $\ca I$ (the right hand side obviously does). This implies that $\si_2^y$ and $(\si_1^y)^2$ have the same image under $\al_2^*$. 
\end{proof}

\subsection{Computing $\al_2^*$ for $G=\SP_{2n}(\re)$}\label{sec:cwSP2n} Here $K=U_n$ and $U=\SP(n)$, and the inclusion $S\hra S'$ is an isomorphism. Identify 
\[H^*(BS')^{W'}\simeq\sym(x_1^2,\ld,x_n^2)\>\>\>\>\text{ and }\>\>\>\>H^*(BS)^W\simeq\sym(y_1,\ld,y_n).\]
The image of $j^*:H^*(BS')^{W'}\ra H^*(BS)^W$ is $\sym(y_1^2,\ld,y_n^2)$. 

\begin{cor}\label{cor:SP2n}
Let $G=\emph{\SP}_{2n}(\re)$. Then $\al_2^*\big(\sum y_i\big)^2\neq0$.
\end{cor}
\begin{proof}
Let $\si_i$ (resp.\ $\om_i$) denote the $i$-th elementary symmetric polynomial in $\{y_1,\ld,y_n\}$ (resp. $\{y_1^2,\ld,y_n^2\}$). Let $\ca I$ denote the ideal $(\om_1,\ld,\om_n)$.

By Theorem \ref{thm:milnor}, to show $\si_1^2=\big(\sum y_i\big)^2$ is not in $\ker\al_2^*$, it is enough to show $\si_1^2\notin\ca I$. Note that $\om_1=\si_1^2-2\si_2$ and that the only elements of $\ca I$ of degree 2 are scalar multiples of $\om_1$. Since $\si_1^2$ is obviously not a multiple of $\si_1^2-2\si_2$, we conclude $\si_1^2\notin\ca I$.
\end{proof}

\subsection{Computing $\al_2^*$ for $G=\SO_{p,q}$}\label{sec:cwSOpq} Here $K=\SO_p\ti\SO_q$ and $U=\SO_{p+q}$. Let $m=[(p+q)/2]$, let $a=[p/2]$, and let $b=[q/2]$. Identify 
\[H^*(BS')\simeq\q[x_1,\ld,x_m]\>\>\>\>\text{ and }\>\>\>\>H^*(BS)\simeq[y_1,\ld,y_p, z_1,\ld,z_q].\]
It is not hard to see that the image of $j^*:H^*(BS')^{W'}\ra H^*(BS)^W$ contains  
$\sym(y_1^2,\ld,y_p^2,z_1^2,\ld,z_q^2)$. 

\begin{cor}\label{cor:SOpq}
Let $p,q\ge2$ and let $G=\emph{\SO}_{p,q}$. Then $\al_2^*\big(\sum y_i^2\big)\neq0$. 
\end{cor}
\begin{proof}
The proof is the same as in Corollary \ref{cor:SP2n}. The image of 
\[H^4(BS')^{W'}\simeq H^4(BG_\co)\ra H^4(BG)\simeq H^4(BS)^W\] is generated by multiples of $\sum y_i^2+\sum z_j^2$, and $\sum y_i^2$ does not have this form.
\end{proof}

\begin{cor}\label{cor:SOpp}
Let $p\ge4$ and let $G=\emph{\SO}_{p,p}$. Then $\al_2^*:H^*(BG)\ra H^*(BG^\de)$ is injective on the linear subgroup generated by $\sum_{i<j}y_i^2y_j^2$ and $\big(\sum y_i^2\big)^2$.
\end{cor}
\begin{proof}
The proof is identical to the proof of Corollary \ref{cor:SUpq2} after replacing $y_i$ and $z_j$ by $y_i^2$ and $z_j^2$. 
\end{proof}

If $p=2,3$, then $a=b=1$ and the polynomial $\big(\sum y_i^2\big)$ in the statement of Corollary \ref{cor:SOpp} is equal to $y_1^4$. In this case we show $\al_2^*(y_1^4)=0$. 
\begin{cor}\label{cor:SOpp2}
Let $G=\emph{\SO}_{2,2}$ or $\emph{\SO}_{3,3}$. Then $\al_2^*(y_1^4)=0$.
\end{cor}
\begin{proof}
If $p=2$ or $p=3$, then $H^*(BG)=\q[y_1^2,z_1^2]$. We show $\al_2^*(y_1^4)=0$. By the computation of $H^*(BS')^{W'}\ra H^*(BS)^W$ above, $\al_2^*(y_1^2+z_1^2)=0$ and $\al_2^*(y_1^2z_1^2)=0$. Then
\[\al_2^*(y_1^4)=\al_2^*(y_1^2)\cdot\al_2^*(y_1^2)=-\al_2^*(y_1^2)\al_2^*(z_1^2)=-\al_2^*(y_1^2z_1^2)=0.\qedhere\]
\end{proof}

\subsection{Computing $\al_2^*$ for $G=\SP_{p,q}$}\label{sec:cwSPpq} Here $K=\SP(p)\ti\SP(q)$ and $U=\SP(p+q)$, and the inclusion $S\hra S'$ is an isomorphism. Identify 
\[H^*(BS')^{W'}\simeq \sym(x_1^2,\ld,x_{p+q}^2)\]
and
\[H^*(BS)^W\simeq\sym(y_1^2,\ld,y_p^2)\ot\sym(z_1^2,\ld,z_q^2).\]
The image of $j^*:H^*(BS')^{W'}\ra H^*(BS)^W$ is $\sym(y_1^2,\ld,y_p^2,z_1^2,\ld,z_q^2)$. We have the following Corollaries \ref{cor:SPpq} and \ref{cor:SPpp} whose proofs are identical to the proofs for Corollaries \ref{cor:SOpq} and \ref{cor:SOpp}, respectively. 

\begin{cor}\label{cor:SPpq}
Let $p,q\ge1$ and let $G=\emph{\SP}_{p,q}$. Then $\al_2^*$ is injective on $\sum y_i^2$.
\end{cor}

\begin{cor}\label{cor:SPpp}
Let $q\ge1$ and $p\ge2$, and let $G=\emph{\SP}_{p,q}$. Then $\al_2^2$ is injective on $\big(\sum y_i\big)^2$. If $p,q\ge2$, then $\al_2^*$ is injective on the linear subgroup generated by $\sum_{i<j}y_i^2y_j^2$ and $\big(\sum y_i^2\big)^2$.
\end{cor}

\subsection{Computing $\al_2^*$ for $G=\SO^*_{2n}$}\label{sec:cwSO*} Here $K=U_n$ and $U=\SO_{2n}$, and the inclusion $S\hra S'$ is an isomorphism. Identify
\[H^*(BS')\simeq\q[x_1,\ld,x_n]\>\>\>\>\>\text{ and }\>\>\>\>\>H^*(BS)\simeq \q[y_1,\ld,y_n].\]
The image of 
\[j^*: H^*(BS')^{W'}\ra H^*(BS)^W=\sym(y_1,\ld,y_n)\]
contains $\sym(y_1^2,\ld,y_n^2)$. We have the following corollary, whose proof is identical to the proof of Corollary \ref{cor:SP2n}. 

\begin{cor}\label{cor:SO*}
Let $G=\emph{\SO}^*_{2n}$. Then $\al_2^*\big(\sum y_i\big)^2\neq0$. 
\end{cor}

\subsection{Computing $\al_2^*$ for $G=\SU^*_{2n}$}\label{sec:cwSU*} Here $K=\SP_n$ and $U=\SU_{2n}$. Identify 
\[H^*(BS')^{W'}\simeq\fr{\sym(x_1,\ld,x_{2n})}{(x_1+\cd+x_{2n})}\]
and
\[H^*(BS)^W\simeq\sym(y_1^2,\ld,y_n^2).\]
It is not hard to see that $j^*:H^*(BS')^{W'}\ra H^*(BS)^W$ is surjective. 

\begin{cor}\label{cor:SU*}
Let $G=\emph{\SU}_{2n}^*$. For $i>0$, the map $\al_2^*:H^i(BG)\ra H^i(BG^\de)$ is zero.
\end{cor}

\subsection{Computing $\al_2^*$ for $G$ a real form of $E_8$}\label{sec:cwE8}\mbox{ }

\vspace{.1in}
\noindent\un{$E_{8(8)}$} Here $\text{Lie}(K)=\mf{so}_{16}$ and $U=E_8$, and the inclusion $S\hra S'$ is an isomorphism. As in Section \ref{sec:background} we identify
\[H^1(S')\simeq\lan J_1,\ld,J_8\ran\>\>\>\>\text{ and }\>\>\>\> H^1(S)\simeq\lan L_1,\ld,L_8\ran.\]
Under $H^1(S')\ra H^1(S)$ we have $J_i\mapsto L_i$. Let $z_i=\tau(J_i)\in H^2(BS')$ and let $y_i=\tau(L_i)\in H^2(BS)$. In Section \ref{sec:background}, we explained that 
\[H^*(BS')^{W'}\simeq\q[I_2,I_8,I_{12}, I_{14}, I_{18}, I_{20}, I_{24}, I_{30}].\]  
Under $j^*:H^4(BS')\ra H^4(BS)$ the polynomial $I_2$ maps to a multiple of $y_1^2+\cd+y_8^2$. By Theorem \ref{thm:milnor} the elements of $H^8(BS)^W$ in the kernel $\al_2^*$ are multiples of 
\[\big(y_1^2+\cd+y_8^2\big)^2=\left(\sum_{i=1}^8y_i^4\right)+2\left(\sum_{1\le i<j\le 8}y_i^2y_j^2\right)\]
This implies the following corollary.
\begin{cor}\label{cor:E88}
Let $G=E_{8(8)}$. With the notation above, $\al_2^*\big(\sum_{i=1}^8y_i^4\big)\neq0$.
\end{cor}

\vspace{.1in}

\noindent\un{$E_{8(-24)}$} Here $K=E_7\ti\SU_2$ and $U=E_8$, and the inclusion $S\hra S'$ is an isomorphism. Identify $H^1(S')\simeq\lan J_1,\ld,J_8\ran$ and 
\[H^1(S)\simeq\lan J_1,\ld,J_6, J_7+J_8\ran\op\lan J_7-J_8\ran.\] Let $z_i=\tau(L_i)\in H^2(BS')$ for $1\le i\le 8$. Let $y_i=\tau(L_i)\in H^2(BS)$ for $1\le i\le 6$,  let $y_7=\tau(J_7+J_8)$ and $y_8=\tau(L_7-L_8)$. In Section \ref{sec:background} we explained that 
\[H^*(BS')^{W'}\simeq\q[I_2,I_8,I_{12}, I_{14}, I_{18}, I_{20}, I_{24}, I_{30}],\]  
and $I_2=30(z_1^2+\cd+z_8^2)$. Under $j^*:H^2(BS')\ra H^2(BS)$, $z_i\mapsto y_i$ for $1\le i\le 6$, $z_7\mapsto \fr{y_7+y_8}{2}$, and $z_8\mapsto \fr{y_7-y_8}{2}$, so  
\[j^*(I_2)= 30 (y_1^2+\cd+y_6^2)+15(y_7^2+y_8^2).\]
By Theorem \ref{thm:milnor}, every element of the kernel $H^4(BE_{8(-24)})\ra H^4\big(B(E_{8(-24)})^\de\big)$ is a scalar multiple of $j^*(I_2)$. This implies the following corollary. 

\begin{cor}\label{cor:E824}
Let $G=E_{8(-24)}$. With the notation above, $\al_2^*(y_8^2)\neq0$.
\end{cor}

\subsection{Computing $\al_2^*$ for $G$ a real form of $E_7$}\label{sec:cwE7} \mbox{ }

\vspace{.1in}

\noindent\un{$E_{7(7)}$} Here $\text{Lie}(K)=\mf{su}_8$ and $U=E_7$, and the inclusion $S\hra S'$ is an isomorphism. We identify $H^1(S')$ as the subspace of $\lan J_1,\ld,J_8\ran$ orthogonal to $J_1+\cd+J_8$, and we identify $H^1(S)$ as the subspace of $\lan L_1,\ld,L_8\ran$ orthogonal to $L_1+\cd+L_8$. Then $H^1(S')\ra H^1(S)$ is the obvious map $J_i\mapsto L_i$. Let $z_i=\tau(J_i)\in H^2(BS')$ and $y_i=\tau(L_i)\in H^2(BS)$. In Section \ref{sec:background} we explained that 
\[H^*(BS')^{W'}\simeq\q[I_2,I_6,I_8,I_{10},I_{12}, I_{14}, I_{18}],\]
and computed $I_2= \fr{3}{4}\left[7\left(\sum z_i^2\right)+2\left(\sum z_iz_j\right)\right]$ for this copy of $E_7\sbs E_8$. Under $j^*:H^*(BS')\ra H^*(BS)$
\[I_2\mapsto \fr{3}{4}\left[ 7(\sum y_i^2)+2(\sum y_iy_j)\right]= -9\left(\sum y_iy_j\right).\]
The equality follows because of the relation $\sum y_i^2=-2\sum y_iy_j$ in $H^*(BS)^W$. By Theorem \ref{thm:milnor}, every element of the kernel of 
\[\al_2^*:H^8\big(BE_{7(7)}\big)\ra H^8\big(B(E_{7(7)})^\de\big)\]
is a scalar multiple of $\left(\sum y_i^2\right)^2$. This implies the following corollary. 

\begin{cor}\label{cor:E77}
Let $G=E_{7(7)}$. With the notation above, $\al_2^*\big(\sum y_i^4\big)\neq0$.
\end{cor}

\vspace{.1in}

\noindent\un{$E_{7(-5)}$} Here $\text{Lie}(K)=\mf{so}_{12}\ti\mf{su}_2$ and $U=E_7$, and the inclusion $S\hra S'$ is an isomorphism. As in Section \ref{sec:background}, we identify
\[H^1(S')\simeq\lan J_1,\ld,J_6, J_7+J_8\ran\>\>\>\>\text{ and }\>\>\>\> H^1(S)\simeq\lan L_1,\ld,L_6\ran\op\lan L_7\ran.\] Under $H^1(S')\ra H^1(S)$ we have 
\[\left\{\begin{array}{clll}
J_i&\mapsto&L_i& 1\le i\le 6\\[1mm]
J_7+J_8&\mapsto&2L_7
\end{array}\right.\]
Let $z_i=\tau(J_i)\in H^2(BS')$ for $1\le i\le 6$ and let $z_7=\tau(J_7+J_8)$. Let $y_i=\tau(L_i)\in H^2(BS)$. In Section \ref{sec:background} we explained that 
\[H^*(BS')^{W'}\simeq\q[I_2,I_6,I_8,I_{10},I_{12}, I_{14}, I_{18}].\]
Under $j^*:H^4(BS')\ra H^4(BS)$, 
\[I_2\mapsto 6 (y_1^2 + y_2^2 + y_3^2 + y_4^2 + y_5^2 + y_6^2+2y_7^2).\]
In particular, by Theorem \ref{thm:milnor} every element of the kernel of 
\[\al_2^*:H^4(BE_{7(-5)})\ra H^4\big(B(E_{7(-5)})^\de\big)\]
is a scalar multiple of $j^*(I_2)$. This implies the following corollary. 
\begin{cor}\label{cor:E75}
Let $G=E_{7(5)}$. With the notation above, $\al_2^*(y_7^2)\neq0$.
\end{cor}

\vspace{.1in}

\noindent\un{$E_{7(-25)}$} Here $\text{Lie}(K)=\mf e_6\ti\mf{so}_2$ and $U=E_7$, and the inclusion $S\hra S'$ is an isomorphism. Identify $H^1(S')\simeq\lan J_1,\ld,J_6, J_7+J_8\ran$ and 
\[H^1(S)\simeq\lan J_1,\ld,J_5, J_6+J_7+J_8\ran\op\lan 2J_6-J_7-J_8\ran.\] Let $z_i=\tau(J_i)\in H^2(BS')$ for $1\le i\le,6$, and let $z_7=\tau(J_7+J_8)$. For $1\le i\le 5$, let $y_i=\tau(J_i)\in H^2(BS)$, and let $y_6=\tau(J_6+J_7+J_8)$ and $y_7=\tau(2J_6-J_7-J_8)$. 

In Section \ref{sec:background} we explained that 
\[H^*(BS')^{W'}\simeq\q[I_2, I_6, I_8, I_{10}, I_{12}, I_{14}, I_{18}],\]
where $I_2=6(z_1^2+\cd+z_6^2)+3z_7^2$. Under $j^*: H^2(BS')\ra H^2(BS)$, $z_i\mapsto y_i$ for $1\le i\le 5$, $z_6\mapsto\fr{y_6-y_7}{3}$, $z_7\mapsto\fr{2y_6+y_7}{3}$, so  
\[j^*(I_2)=6 (y_1^2 +  y_2^2 + y_3^2 + y_4^2 + y_5^2) + 2 y_6^2 + y_7^2.\]
By Theorem \ref{thm:milnor}, $j^*(I_2)$ generates the kernel of 
\[\al_2^*:H^4\big(BE_{7(-25)}\big)\ra H^4\big(B(E_{7(-25)})^\de\big).\]
\noindent In particular, we have the following corollary.
\begin{cor}\label{cor:E725}
Let $G=E_{7(-25)}$. With the notation above, $\al_2^*(y_7^2)\neq0$.
\end{cor}

\subsection{Computing $\al_2^*$ for $G$ a real form of $E_6$}\label{sec:cwE6}\mbox{ }

\vspace{.1in}
\noindent\un{$E_{6(6)}$} Here $\text{Lie}(K)=\mf{sp}_4$ and $U=E_6$. Identify $H^1(S)\simeq\land L_1,\ld,L_4\ran$ and 
\[H^1(S')\sbs\lan J_1,\ld,J_6,J_7-J_8\ran\]
as the subspace orthogonal to the roots in (\ref{eqn:rootsE6}). Under $\mf{sp}_4\hra\mf e_6\hra\mf e_8$ the Cartan subalgebra $\mf h(\mf{sp}_4)\sbs\mf{sp}_4$ is contained in the image of the Cartan subalgebra for $\mf{su}_6\ti\mf{su}_2\sbs\mf{su}_8$. Under $\mf h^*(\mf e_8)\ra \mf h^*(\mf{sp}_4)$, we have 
\[\begin{array}{ccrc}
J_i&\mapsto& L_i&1\le i\le 3\\
J_{i+3}&\mapsto& -L_i&1\le i\le 3\\
J_7&\mapsto& L_4\\
J_8&\mapsto &-L_4
\end{array}\]
It is then easy to determine the restriction of this map to $\mf h^*(\mf e_6)$. 

Let $z_i=\tau(J_i)$ for $1\le i\le 6$ and let $z_7=\tau(J_7-J_8)$ in $H^2(BS')$. Let $y_i=\tau(L_i)\in H^2(BS)$ for $1\le i\le 4$. From Section \ref{sec:background}, $H^*(BS')^{W'}\simeq\q[I_2,I_5, I_6, I_8, I_9, I_{12}]$, where \[I_2=5(z_1^2+\cd+z_6^2)+3z_7^2-2\left(\sum_{1\le i<j\le 6}z_iz_j\right).\]
Under $j^*: H^*(BS')\ra H^*(BS)$, 
\[I_2\mapsto 12(y_1^2+y_2^2+y_3^2+y_4^2).\]
By Theorem \ref{thm:milnor}, $j^*(I_2)$ and $j^*(I_2)^2$ generate (as a vector space) the kernel of 
\[\al_2^*:H^k\big(BE_{6(6)}\big)\ra H^k\big(B(E_{6(6)})^\de\big)\]
for $k=4,8$. In particular, we have the following corollary. 
\begin{cor}\label{cor:E66}
Let $G=E_{6(6)}$. With the notation above, $\al_2^*\big(\sum y_i^4\big)\neq0$.
\end{cor}

\vspace{.1in}

\noindent\un{$E_{6(2)}$} Here $\text{Lie}(K)=\mf{su}_6\ti\mf{su}_2$ and $U=E_6$, and the inclusion $S\hra S'$ is an isomorphism. We identify $H^1(S')\sbs\lan J_1,\ld,J_6,J_7-J_8\ran$ as the subspace of $\lan J_1,\ld,J_8\ran$ orthogonal to the roots in (\ref{eqn:rootsE6}). We identify $H^1(S)$ as the subspace of $\lan L_1,\ld,L_6\ran\op \lan L_7\ran$ orthogonal to $L_1+\cd+L_6$. The map $H^1(S')\ra H^1(S)$ is given by $J_i\mapsto L_i$ for $1\le i\le 6$ and $J_7-J_8\mapsto 2L_i$. 

Let $z_i=\tau(J_i)\in H^2(BS')$ for $1\le i\le 6$ and let $z_7=\tau(J_7-J_8)$. Let $y_i=\tau(L_i)\in H^2(BS)$ for $1\le i\le 7$. In Section \ref{sec:background}, we explained that 
\[H^*(BS')^{W'}\simeq\q[I_2,I_5,I_6,I_8,I_9,I_{12}],\]
where 
\[I_2=5(z_1^2+\cd+z_6^2)+3\>z_7^2-2\left(\sum_{1\le i<j\le 6}z_iz_j\right).\] Under $j^*:H^*(BS')\ra H^*(BS)$,
\[I_2\mapsto 5(y_1^2+\cd+y_6^2)+12\>y_7^2-2\left(\sum_{1\le i<j\le 6}y_iy_j\right).\]
Since $\sum y_i^2=-2\sum y_iy_j$ in $H^*(B\SU_6)$, we can re-write this polynomial
\[j^*(I_2)=6\big(y_1^2+\cd+y_6^2+2y_7^2\big).\]
By Theorem \ref{thm:milnor}, the kernel of 
\[\al_2^*:H^4(BE_{6(2)})\ra H^4\big(B(E_{6(2)})^\de\big)\]
is generated by scalar multiples of $j^*(I_2)$. In particular, we have the following corollary. 

\begin{cor}\label{cor:E62}
Let $G=E_{6(2)}$. With the notation above, $\al_2^*(y_7^2)\neq0$.
\end{cor}

\vspace{.1in}

\noindent\un{$E_{6(-14)}$} Here $\text{Lie}(K)=\mf{so}_{10}\ti\mf{so}_2$ and $U=E_6$, and the inclusion $S\hra S'$ is an isomorphism. As in Section \ref{sec:background} we identify 
\[H^1(S')\simeq\lan J_1,\ld,J_5, J_6+J_7+J_8\ran\>\>\>\>\text{ and }\>\>\>\> H^1(S)\simeq\lan L_1,\ld,L_5\ran\op\lan L_6\ran.\]
Under $H^1(S')\ra H^1(S)$ we have $J_i\mapsto L_i$ for $1\le i\le 5$ and $J_6+J_7+J_8\mapsto L_6$. Let $z_i=\tau(J_i)\in H^2(BS')$ for $1\le i\le 5$ and let $w=\tau(J_6+J_7+J_8)$. Let $y_i=\tau(L_i)\in H^2(BS)$. In Section \ref{sec:background}, we explained that 
\[H^*(BS')^{W'}\simeq\q[I_2,I_5,I_6,I_8,I_9,I_{12}].\]
Under $H^4(BS')\ra H^4(BS)$, 
\[I_2\mapsto 6 (y_1^2 + y_2^2 +y_3^2 + y_4^2 + y_5^2) + 3\> y_6^2.\]
In particular, by Theorem \ref{thm:milnor} every element of the kernel of 
\[\al_2^*:H^4(BE_{6(-14)})\ra H^4\big(B(E_{6(-14)})^\de\big)\] is a scalar multiple of $j^*(I_2)$. This implies the following corollary. 

\begin{cor}\label{cor:E614}
Let $G=E_{6(-14)}$. With the notation above, $\al_2^*(y_6^2)\neq0$.
\end{cor}

\vspace{.1in}

\noindent\un{$E_{6(-26)}$} Here $K=F_4$ and $U=E_6$. Identify 
\[H^1(S)\simeq\lan L_1,\ld,L_4\ran\>\>\>\>\>\text{ and }\>\>\>\>\>H^1(S')=\lan J_1,\ld,J_5,J_6+J_7+J_8\ran.\]Let $z_i=\tau(J_i)\in H^2(BS')$ for $1\le i\le 5$ and let $z_6=\tau(J_6+J_7+J_8)$. For $1\le i\le 4$, let $y_i=\tau(L_i)\in H^2(BS)$.

Under $H^*(BS')\ra H^*(BS)$ we have $z_5,z_6\mapsto 0$ and $z_i\mapsto y_i$ for $i=1,\ld,4$. By the Equations (\ref{eqn:invarF4}) and (\ref{eqn:invarE6}), it is easy to see that $j^*$ is surjective. 
%This is immediate from the equations for the invariant polynomials given in (\ref{eqn:invarF4}) and (\ref{eqn:invarE6}). 

\begin{cor}\label{cor:E626}
Let $G=E_{6(-26)}$. For $i>0$, the map $\al_2^*:H^i\big(BG\big)\ra H^i\big(BG^\de\big)$ is zero. 
\end{cor}

\subsection{Computing $\al_2^*$ for $G$ a real form of $F_4$}\label{sec:cwF4}\mbox{ }

\vspace{.1in}

\noindent\un{$F_{4(4)}$}. Here $\text{Lie}(K)=\mf{su}_2\ti\mf{sp}_3$ and $U=F_4$, and the inclusion $S\hra S'$ is an isomorphism. As in Section \ref{sec:background} we identify $H^1(S')\simeq\lan J_1,\ld,J_4\ran$ and we identify $H^1(S)\simeq\lan L_1\ran\op\lan L_2,L_3,L_4\ran$. Let $z_i=\tau(J_i)\in H^*(BS')$ and $y_i=\tau(L_i)\in H^*(BS)$ for $i=1,\ld,4$. 
In Section \ref{sec:isotropyF4}, we described the isomorphism $H^*(S')\ra H^*(S)$, and so under $j^*:H^*(BS')\ra H^*(BS)$, 
\[j:\left\{\begin{array}{llll}z_1&\mapsto& y_1+y_2\\[1mm]z_2&\mapsto&y_2-y_1\\[1mm]z_3&\mapsto& y_3+y_4\\[1mm]z_4&\mapsto& y_3-y_4\end{array}\right.\]
In Section \ref{sec:background} we explained that $H^*(BS')^{W'}\simeq\q[I_2,I_6,I_8,I_{12}]$, where 
\[I_2=3(z_1^2+z_2^2+z_3^2+z_4^2).\] Under $j^*:H^*(BS')\ra H^*(BS)$, 
\[\begin{array}{lll}
I_2/3&\mapsto &(y_1+y_2)^2+(y_2-y_1)^2+(y_3+y_4)^2+(y_3-y_4)^2\\[2mm]
&=&2(y_1^2+y_2^2+y_3^2+y_4^2).\end{array}\]
By Theorem \ref{thm:milnor} the elements of $H^4(BS)^W$ that are in the kernel of $\al_2^*$ are all scalar multiples of $y_1^2+\cd+y_4^2$. The following corollary will be used in Section \ref{sec:computation}. 
\begin{cor}\label{cor:F44}
Let $G=F_{4(4)}$. With the notation above, $\al_2^*(y_1^2)\neq0$.
\end{cor}
\vspace{.1in} 

\noindent\un{$F_{4(-20)}$}. Here $\text{Lie}(K)=\mf{so}_9$ and $U=F_4$, and the inclusion $S\hra S'$ is an isomorphism. As in Section \ref{sec:background}, 
\[H^1(S')\simeq \lan J_1,\ld,J_4\ran\>\>\>\>\text{ and }\>\>\>\>H^1(S)\simeq\lan L_1,\ld,L_4\ran.\] Under $H^1(S')\ra H^1(S)$ we have $J_i\mapsto L_i$. Let $z_i=\tau(J_i)\in H^2(BS')$ and $y_i=\tau(L_i)\in H^2(BS)$ for $i=1,\ld,4$. In Section \ref{sec:background} we explained that $H^*(BS')^{W'}\simeq\q[I_2,I_6,I_8,I_{12}]$, where $I_2=3(z_1^2+z_2^2+z_3^2+z_4^2)$.
By Theorem \ref{thm:milnor}, the elements of $H^8(BS)^W$ that are in the kernel of $\al_2^*$ are all scalar multiples of 
\[j^*(I_2)^2=9\cdot\big[\big(\sum_{i=1}^4 y_i^4\big)+2\big(\sum_{1\le i<j\le4} y_i^2y_j^2\big)\big].\] This implies the following corollary. 
\begin{cor}\label{cor:F420}
Let $G=F_{4(-20)}$. With the notation above, $\al_2^*\big(\sum_{1\le i<j\le 4} y_i^2y_j^2\big)\neq0$.
\end{cor}

\subsection{Computing $\al_2^*$ for $G$ a real form of $G_2$}\label{sec:cwG2}\mbox{ }

\vspace{.1in}
\noindent\un{$G_{2(2)}$} Here $K=\SU_2\ti\SU_2$ and $U=G_2$, and the inclusion $S\hra S'$ is an isomorphism. $H^1(S')\simeq\lan J_1,J_2\ran$ and $H^1(S)\simeq\lan L_1,L_2\ran$, and under $H^1(S')\ra H^1(S)$ 
\[\left\{\begin{array}{lcc}
J_1&\mapsto& 2L_1\\[1mm]
J_2&\mapsto &-3L_1+L_2
\end{array}\right.\]
Let $z_i=\tau(J_i)$ in $H^2(BS')$ for $i=1,2$ and $y_i=\tau(L_i)$ in $H^2(BS)$ for $j=1,2$. 

In Section \ref{sec:background} we explained that $H^*(BS')^{W'}\simeq \q[I_2,I_{6}]$, where 
\[I_2=2(3z_1^2+3z_1z_2+z_2^2).\]%x_1^2+(x_1+x_2)^2+(2x_1+x_2)^2.\]
Under $j^*:H^*(BS')\ra H^*(BS)$, $I_2\mapsto 2(3y_1^2+y_2^2)$. By Theorem \ref{thm:milnor}, $3y_1^2+y_2^2$ generates the kernel of $\al_2^*:H^4(BG_{2(2)})\ra H^4\big(B(G_{2(2)})^\de\big)$. In particular, we have the following corollary.
\begin{cor}\label{cor:G2}
Let $G=G_{2(2)}$. With the notation above, $\al_2^*(y_1^2)\neq0$.
\end{cor}

\section{Computations}\label{sec:computation}

Recall that $G$ is a simple Lie group from Table 2 and $\Ga\sbs G$ is a cocompact lattice. Our goal in this section is to find nonzero elements in the image of the composition
\[H^*(B\homeo(S^{n-1}))\xra{\al_1^*}H^*(BG)\xra{\al_2^*}H^*(BG^\de)\xra{\al_3^*}H^*(B\Ga).\]
In Section \ref{sec:isotropy1} we showed that $\al_1^*(q_i)=p_i(\iota)=(-1)^ic_{2i}(\iota_\co)$, and we computed the polynomial $c(\iota_\co)$. The kernel of $\al_2^*$ is generated as an ideal by the image of $i^*: H^{>0}(BG_\co)\ra H^{>0}(BG)$ by Theorem \ref{thm:milnor}, and $i^*$ was computed in Section \ref{sec:cw}. Finally, $\al_3^*$ is injective by the following proposition. 
\begin{prop}\label{prop:step3}
Let $\Ga\sbs G$ be a cocompact lattice. Then the image of $H^*(BG)\ra H^*(BG^\de)$ injects into $H^*(\Ga)$. 
\end{prop}
The proof of Proposition \ref{prop:step3} uses a transfer argument and can be found in Section 2.3 of \cite{bg}. We can thus combine the computations from Sections \ref{sec:isotropy1} and \ref{sec:cw} to determine if $p_i(\Ga\bs G/K)\neq0$. 

\subsection{Pontryagin classes for $\Sl_n(\re)$-manifolds}

\begin{thm}
Let $G=\emph{\Sl}_n(\re)$ and let $\Ga\sbs G$ a cocompact lattice. Then $p_i(\Ga\bs G/K)=0$ for $i>0$.
\end{thm}
\begin{proof}
In Section \ref{sec:isotropySLn} we computed the total Chern class $c(\iota_\co)$ of the isotropy representation. By Equations (\ref{eqn:chernSL}) and (\ref{eqn:chernSL2}),  $c(\iota_\co)$ is a symmetric polynomial in $\{y_1^2, \ld,y_k^2\}$, where $k=[n/2]$. This implies that $c(\iota_\co)$ is in the image of 
\[H^*(B\Sl_n(\co))\ra H^*(B\Sl_n(\re)),\] 
by Section \ref{sec:cwSLn}. Then $\al_2^*\big(c(\iota_\co)\big)=0$ by Theorem \ref{thm:milnor}. Hence
\[p_i(M)=\al_3^*\al_2^*\al_1^*(q_i)=(-1)^i\>\al_3^*\al_2^*\big(c(\iota_\co)\big)=0.\qedhere\]
%Since $p_i(\iota)=(-1)^ic_{2i}(\iota_\co)$ and $p_i(M)=\al_3^*\al_2^*\>p_i(\iota)$, the Pontryagin classes of $M$ are trivial.
\end{proof}

\subsection{Pontryagin classes for $\SU_{p,q}$-manifolds} 

\begin{thm}\label{thm:SUpq}
Let $p,q\ge1$ and $(p,q)\neq(1,1)$. Let $G=\emph{\SU}_{p,q}$ and let $\Ga\sbs G$ a cocompact lattice. Then $p_1(\Ga\bs G/K)\neq0$.
\end{thm}
\noindent Note that $\SU_{1,1}\simeq\SO_{2,1}$. This case is covered in Theorem \ref{thm:SOp1}.
\begin{proof}[Proof of Theorem \ref{thm:SUpq}]
From Equation (\ref{eqn:chernSU}), one computes %for the total Chern class of the isotropy representation, one computes 
\[p_1(\iota)=q\left(\sum y_i^2\right)+p\left(\sum z_j^2\right)-2\sum y_iz_j.
\]
Using the relation $\sum y_i+\sum z_j=0$ in $H^*(BS)^W\simeq H^*(BG)$, and the relation 
\[\al_2^*\left(\sum y_i^2+\sum z_j^2\right)=0\]
from Section \ref{sec:cwSUpq}, it follows that
\[\al_2^*\>p_1(\iota)=(q-p+2)\>\al_2^*\left(\sum y_i\right)^2+2(p-q)\>\al_2^*\left(\sum y_iy_j\right).\]
If $q=1$, then $\al_2^*\big(\sum y_i\big)^2=\al_2^*\big(\sum y_iy_j\big)$ by Corollary \ref{cor:SUpq} and so
\[\al_2^*\>p_1(\iota)=(p+1)\>\al_2^*\big(\sum y_iy_j\big),\]
which is nonzero (again, see Corollary \ref{cor:SUpq}). 

If $q\ge2$, then $\al_2^*\big(\sum y_i\big)^2$ and $\al_2^*\big(\sum y_iy_j\big)$ are linearly independent by Corollary \ref{cor:SUpq2}, and it follows that $\al_2^*\>p_1(\iota)\neq0$. Hence $p_1(M)=\al_3^*\al_2^*\>p_1(\iota)\neq0$ by Proposition \ref{prop:step3}.
\end{proof}

\subsection{Pontryagin classes for $\SP_{2n}(\re)$-manifolds}

\begin{thm} Fix $n\ge2$. 
Let $G=\emph{\SP}_{2n}(\re)$ and let $\Ga\sbs G$ a cocompact lattice. Then $p_1(\Ga\bs G/K)\neq0$.
\end{thm}

\begin{proof}
From Equation (\ref{eqn:chernSP}), one computes% for the total Chern class of the isotropy representation one computes
\[p_1(\iota)=(n+3)\sum y_i^2+2\sum y_iy_j.\]
By Section \ref{sec:cwSP2n}, $\al_2^*\big(\sum y_i^2\big)=0$. Using the relation $\big(\sum y_i\big)^2=\sum y_i^2+2\sum y_iy_j$, 
\[\al_2^*\>p_1(\iota)=2\>\al_2^*\big(\sum y_iy_j\big)=\al_2^*\big(\sum y_i\big)^2,\]
which nonzero by Corollary \ref{cor:SP2n}. Hence $p_1(M)=\al_3^*\al_2^*\>p_1(\iota)\neq0$ by Proposition \ref{prop:step3}. 
\end{proof}

\subsection{Pontryagin classes for $\SO_{p,q}$-manifolds} 

\begin{thm}\label{thm:SOpq1}
Let $p,q\ge2$. Let $G=\emph{\SO}_{p,q}$ and let $\Ga\sbs G$ a cocompact lattice. Then $p_1(\Ga\bs G/K)\neq0$ if and only if $p\neq q$.
\end{thm}

\begin{thm}\label{thm:SOpq2}
Fix $p\ge4$. 
Let $G=\emph{\SO}_{p,p}$ and let $\Ga\sbs G$ a cocompact lattice. Then $p_2(\Ga\bs G/K)\neq0$. 
\end{thm}

\begin{thm}\label{thm:SOp1}
Let $G=\emph{\SO}_{p,1}$ or $\emph{\SO}_{2,2}$ or $\emph{SO}_{3,3}$. Let $\Ga\sbs G$ a cocompact lattice. Then $p_i(\Ga\bs G/K)=0$ for $i>0$.
\end{thm}

\begin{proof}[Proof of Theorem \ref{thm:SOpq1}]
Let $a=[p/2]$ and $b=[q/2]$. 

\vspace{.1in} 
\noindent{\bf Case 1.} Assume $p=2a$ and $q=2b$ are both even. From Equation (\ref{eqn:chernSO1}) one computes
\[p_1(\iota)=2b\big(\sum y_i^2\big)+2a\big(\sum z_j^2\big).\]
By Section \ref{sec:cwSOpq}, $\al_2^*\big(\sum y_i^2\big)+\al_2^*\big(\sum z_j^2\big)=0$, and so 
\[\al_2^*\>p_1(\iota)=2(b-a)\>\al_2^*\big(\sum y_i^2\big).\]
By Corollary \ref{cor:SOpq}, $\al_2^*\big(\sum y_i^2\big)\neq0$. The assumptions that $p\neq q$ and that $p,q$ are both even imply that $\al_2^*\>p_1(\iota)\neq0$. Hence $p_1(M)=\al_3^*\al_2^*\>p_1(\iota)\neq0$ by Proposition \ref{prop:step3}. 

\vspace{.1in}

\noindent {\bf Case 2.} Assume $p=2a$ and $q=2b+1$. From Equation (\ref{eqn:chernSO2}) one computes 
\[p_1(\iota)=(2b+1)\big(\sum y_i^2\big)+2a\big(\sum z_j^2\big).\]
Similar to Case 1, 
\[\al_2^*\>p_1(\iota)= (2b-2a+1)\>\al_2^*\big(\sum y_i^2\big),\]
which is nonzero and so $p_1(M)=\al_3^*\al_2^*\>p_1(\iota)\neq0$. 

\vspace{.1in}

\noindent {\bf Case 3.} Assume $p=2a+1$ and $q=2b$. Using Equation (\ref{eqn:chernSO3}), in an entirely similar fashion to Case 2, 
\[\al_2^*\>p_1(\iota)=(2b-2a-1)\>\al_2^*\big(\sum y_i^2\big),\]
which is nonzero, and this implies $p_1(M)\neq0$. 

\vspace{.1in}

\noindent{\bf Case 4.} Assume $p=2a+1$ and $q=2b+1$. From Equation (\ref{eqn:chernSO4}) one computes
\[p_1(\iota)=(2b+1)\big(\sum y_i^2\big)+(2a+1)\big(\sum z_j^2\big).\]
Then 
\[\al_2^*\>p_1(\iota)=2(b-a)\>\al_2^*\big(\sum y_i^2\big),\]
which is nonzero since $p\neq q$ and $p$ and $q$ are both odd. Then $p_1(M)\neq0$ in this case.
\end{proof}

\begin{proof}[Proof of Theorem \ref{thm:SOpq2}]We separate the cases when $p$ is even and when $p$ is odd. 

\vspace{.1in} 
\noindent{\bf Case 1.} Assume $p=2a$ is even. From Equation (\ref{eqn:chernSO1}) one computes
\[p_2(\iota)={2a\choose 2}\sum y_i^4+{2a\choose 2}\sum z_j^4+(2a)^2\sum y_i^2y_k^2+(2a)^2\sum z_j^2z_\ell^2+(4a^2-6)\sum y_i^2z_j^2
\]
By the computation in Section \ref{sec:cwSOpq}, 
\[\al_2^*\big(\sum y_i^4+\sum z_j^4\big)=0\>\>\>\>\>\text{ and }\>\>\>\>\>\al_2^*\big(\sum y_i^2y_k^2+\sum z_j^2z_\ell^2+\sum y_i^2z_j^2\big)=0.\]
From these relations, it follows that 
\begin{equation}\label{eqn:pontSOpp1}\al_2^*\>p_2(\iota)=-6\>\al_2^*\big(\sum y_i^2z_j^2\big).\end{equation}
From the fact that $\al_2^*\big(\sum y_i^2+\sum z_j^2\big)^2=0$, it follows that 
\begin{equation}\label{eqn:pontSOpp2}\al_2^*\big(\sum y_i^2z_j^2\big)=\al_2^*\big(\sum y_i^2\big)^2.\end{equation}
By Corollary \ref{cor:SOpp}, $\al_2^*\big(\sum y_i^2\big)^2\neq0$ for $p\ge4$, so we conclude from (\ref{eqn:pontSOpp1}) and (\ref{eqn:pontSOpp2}) that $\al_2^*\>p_2(\iota)\neq0$. Then $p_2(M)=\al_3^*\al_2^*\>p_2(\iota)\neq0$ by Proposition \ref{prop:step3}. 

\vspace{.1in}
\noindent{\bf Case 2.} Assume $p=2a+1$ is odd. The total Chern class of the isotropy representation is given in Equation (\ref{eqn:chernSO4}). 
Let $A=\prod_{i=1}^a(1-y_i^2)\prod_{j=1}^a(1-z_j^2)$. From the computation in Section \ref{sec:cwSOpq} (combined with Theorem \ref{thm:milnor}), it is immediate that $\al_2^*(A)=1$. Then the computation of $\al_2^*\>p_2(\iota)$ is the exact same is in Case 1. Again we conclude $p_2(M)\neq0$.
\end{proof}

\begin{proof}[Proof of Theorem \ref{thm:SOp1}] If $G=\SO_{p,1}$, then $b=0$ and the total Chern class is a symmetric polynomial in $\{y_1^2,\ld,y_a^2\}$. In Section \ref{sec:cwSOpq}, we saw that all such polynomials are in the image of $H^*(BG_\co)\ra H^*(BG)$, and hence in the kernel of $\al_2^*$ by Theorem \ref{thm:milnor}. This implies $p_i(M)=0$ for all $i\ge1$. 

For $G=\SO_{2,2}$, the dimension of $M=\Ga\bs G/K$ is 4, and so the only Pontryagin class that could be nonzero is $p_1(M)$. However, $p_1(M)=0$ by Theorem \ref{thm:SOpq1}. 

For $G=\SO_{3,3}$, the dimension of $M=\Ga\bs G/K$ is 9. Thereom \ref{thm:SOpq1} implies that $p_1(M)=0$. The proof of Theorem \ref{thm:SOpq2} gives
\[p_2(M)=\al_3^*\al_2^*\>p_2(\iota)=-6\>\al_3^*\al_2^*\big(\sum y_i^2\big)^2.\]
By Corollary \ref{cor:SOpp2}, $\al_2^*\big(\sum y_i^2\big)^2=0$, so $p_2(M)=0$. 
\end{proof}

\subsection{Pontryagin classes for $\SP_{p,q}$-manifolds}

It is worth noting the similarity between $\SP_{p,q}$ and $\SO_{2p,2q}$. On the level of cohomology, 
\[H^*(B\SP_{p,q})\simeq\sym(y_1^2,\ld,y_p^2)\ot\sym(z_1^2,\ld,z_q^2),\]
while
\[H^*(B\SO_{2p,2q})\simeq\big\lan\sym(y_1^2,\ld,y_p^2), y_1\cd y_p\big\ran\ot\big\lan \sym(z_1^2,\ld,z_q^2), z_1\cd z_q\big\ran.\]
Note also from Sections \ref{sec:isotropySOpq} and \ref{sec:isotropySPpq}, the weights of the isotropy representation for both are $\pm y_i\pm z_j$ for $i=1,\ld, p$ and $j=1,\ld q$. Then in both cases the total Chern class of the isotropy representation is
\[c(\iota_\co)=\prod(1+y_i+z_j)(1+y_i-z_j)(1-y_i+z_j)(1-y_i-z_j).\]
Finally, by Sections \ref{sec:cwSOpq} and \ref{sec:cwSPpq}, for $G=\SP_{p,q}$ or $G=\SO_{2p,2q}$, 
\[\im\left[j^*:H^*(BG_\co)\ra H^*(BG)\right]=\sym(y_1^2,\ld,y_p^2,z_1^2,\ld,z_q^2),\]
so the relations used to compute $\al_2^*\>p_i(\iota)$ for $\SO_{2p,2q}$ also compute $\al_2^*\>p_i(\iota)$ for $\SP_{p,q}$. From these observations and from the proofs of  Theorems \ref{thm:SOpq1} and \ref{thm:SOpq2}, we have the following two theorems. (One should replace the use of Corollaries \ref{cor:SOpq} and \ref{cor:SOpp} in the proofs of Theorems \ref{thm:SOpq1} and \ref{thm:SOpq2} by Corollaries \ref{cor:SPpq} and \ref{cor:SPpp}.)
\begin{thm}
Let $p,q\ge1$ such that $(p,q)\neq(1,1)$. Let $G=\emph{\SP}_{p,q}$ and let $\Ga\sbs G$ a cocompact lattice. Then $p_1(\Ga\bs G/K)\neq0$ if and only if $p\neq q$. 
\end{thm}

\noindent Note that $\SP_{1,1}\simeq\SO_{4,1}$, which is treated in Theorem \ref{thm:SOp1}. 
\begin{thm}
Fix $p\ge2$. Let $G=\emph{\SP}_{p,p}$ and let $\Ga\sbs G$ a cocompact lattice. Then $p_2(\Ga\bs G/K)\neq0$. 
\end{thm}

\subsection{Pontryagin classes for $\SO^*_{2n}$-manifolds}
\begin{thm}\label{thm:pontSO} Fix $n\ge3$. 
Let $G=\emph{\SO}^*_{2n}$ and let $\Ga\sbs G$ a cocompact lattice. Then $p_1(\Ga\bs G/K)\neq0$. 
\end{thm}
\begin{rmk}We only consider $n\ge3$ because $\SO^*_2\simeq\co$ and $\SO^*_{4}$ is not simple.  \end{rmk}

\begin{proof}
From Equation (\ref{eqn:chernSO*}), one computes% for the total Chern class for the isotropy representation, one computes
\[p_1(\iota)=(n-1)\sum y_i^2+2\sum y_iy_j.\]
By the computation in Section \ref{sec:cwSO*}, $\al_2^*\big(\sum y_i^2\big)=0$. Using the relation $\big(\sum y_i\big)^2=\big(\sum y_i^2\big)+2\sum y_iy_j$, it follows that
\[\al_2^*\>p_1(\iota)=2\>\al_2^*\big(\sum y_iy_j\big)=\al_2^*\big(\sum y_i\big)^2,\]
which nonzero by Corollary \ref{cor:SO*}. Then $p_1(M)=\al_3^*\al_2^*\>p_1(\iota)\neq0$ by Proposition \ref{prop:step3}. 
\end{proof}

\subsection{Pontryagin classes for $\SU^*_{2n}$-manifolds}

\begin{thm}
Let $G=\emph{\SU}^*_{2n}$ and let $\Ga\sbs G$ a cocompact lattice. Then $p_i(\Ga\bs G/K)=0$ for $i>0$.
\end{thm}
\begin{proof}
By Corollary \ref{cor:SU*}, $\al_2^*$ is zero in positive degrees. 
\end{proof}

\subsection{Pontryagin classes for $E_8$-manifolds}

\begin{thm}
Let $G=E_{8(8)}$ and let $\Ga\sbs G$ be a cocompact lattice. Then $p_1(\Ga\bs G/K)=0$ and $p_2(\Ga\bs G/K)\neq0$. 
\end{thm}
\begin{proof}
From Equation (\ref{eqn:chernE88}), one computes %for the total Chern class of the isotropy representation, one computes
\[p_1(\iota)=16(y_1^2+\cd+y_8^2)\]
and 
\[p_2(\iota)=126\left(\sum_{i=1}^8y_i^4\right)+244\left(\sum_{1\le i<j\le 8}y_i^2y_j^2\right).\]
By Theorem \ref{thm:milnor} and the computation of Section \ref{sec:cwE8},  $\sum_{i=1}^8y_i^2$ and $\left(\sum_{i=1}^8y_i^2\right)^2$ generate the kernel of 
\[\al_2^*: H^k(BE_{8(8)})\ra H^k\big(B(E_{8(8)})^\de\big)\]
for $k=4,8$. Then $\al_2^*\>p_1(\iota)=0$ and 
\[\al_2^*\>p_2(\iota)= 4\>\al_2^*\left(\sum y_i^4\right)+122\>\underbrace{\al_2^*\left[\big(\sum y_i^2\big)^2\right]}_{=0}
\]
By Corollary \ref{cor:E88}, $\al_2^*\>p_2(\iota)=4\>\al_2^*\left(\sum y_i^4\right)\neq0$, and so $p_2(M)=\al_3^*\al_2^*\>p_2(\iota)\neq0$ by Proposition \ref{prop:step3}. 
\end{proof}

\begin{thm}
Let $G=E_{8(-24)}$ and let $\Ga\sbs G$ be a cocompact lattice. Then $p_1(\Ga\bs G/K)\neq0$.
\end{thm}
\begin{proof}
From Equation (\ref{eqn:chernE824}), one computes % for the total Chern class for the isotropy representation, one computes
\[p_1(\iota)=12(y_1^2+\cd+y_6^2)+6y_7^2+14y_8^2.\]
By Theorem \ref{thm:milnor} and the computation of Section \ref{sec:cwE8}, $2(y_1^2+\cd+y_6^2)+y_7^2$ generates the kernel of 
\[\al_2^*:H^4(BE_{8(-24)})\ra H^4\big(B(E_{8(-24)})^\de\big).\]
Then 
\[\al_2^*\>p_1(\iota)=14\>\al_2^*(y_8^2)\]
is nonzero by Corollary \ref{cor:E824}, and so $p_1(M)=\al_3^*\al_2^*\>p_1(\iota)\neq0$ by Proposition \ref{prop:step3}. 
\end{proof}

\subsection{Pontryagin classes for $E_7$-manifolds}

\begin{thm}
Let $G=E_{7(7)}$ and let $\Ga\sbs G$ be a cocompact lattice. Then $p_1(\Ga\bs G/K)=0$ and $p_2(\Ga\bs G/K)\neq0$.
\end{thm}
\begin{proof}
From Equation (\ref{eqn:chernE77}), one computes % for the total Chern class for the isotropy representation, one computes
\[p_1(\iota)=595\left(\sum y_i^2\right)+1210\left(\sum y_iy_j\right),\]
and 
\[\begin{array}{ll}
p_2(\iota)=&52\>360\left(\sum y_i^4\right)+220\>660\left(\sum y_i^3y_j\right)+336\>790\left(\sum y_i^2y_j^2\right)\\[2mm]&+\>684\>810\left(\sum y_i^2y_jy_k\right)+1\>392\>444\left(\sum y_iy_jy_ky_\ell\right).\end{array}\]
By Theorem \ref{thm:milnor} and the computation of Section \ref{sec:cwE7}, for $k=4,8$, the kernel of 
\[\al_2^*:H^*(BS)^W\simeq H^k(BE_{7(7)})\ra H^k\big(B(E_{7(7)})^\de\big)\]
is generated by $\left(\sum y_i^2\right)$ and $\left(\sum y_i^2\right)^2$. Since $\sum y_i^2$ and $\sum y_iy_j$ are linearly dependent in $H^*(B\SU_8)$, it follows that $\al_2^*\>p_1(\iota)=0$. Using relations among the symmetric polynomials and that $\al_2^*\big(\sum y_i^2\big)^2=0$, we find that 
\[\al_2^*\>p_2(\iota)=-348\>109\>\al_2^*\big(\sum y_i^4\big),\]
which is nonzero by Corollary \ref{cor:E77}. Then $p_2(M)=\al_3^*\al_2^*\>p_2(\iota)\neq0$ by Proposition \ref{prop:step3}. 
\end{proof}

\begin{thm}
Let $G=E_{7(-5)}$ and let $\Ga\sbs G$ be a cocompact lattice. Then $p_1(\Ga\bs G/K)\neq0$. 
\end{thm}
\begin{proof}
From Equation (\ref{eqn:chernE75}), one computes % for the total Chern class of the isotropy representation, one computes
\[p_1(\iota)=8\big(\sum_{i=1}^7y_i^2\big).\]
By Theorem \ref{thm:milnor} and the computation of Section \ref{sec:cwE7}, the kernel of 
\[\al_2^*:H^4(BE_{7(-5)})\ra H^4\big(B(E_{7(-5)})^\de\big)\]
consists of scalar multiples of $I=y_1^2 +\cd + y_6^2+2 y_7^2$, and so 
\[\al_2^*\>p_1(\iota)=8\>\al_2^*(I-y_7^2)=-8\>\al_2^*(y_7^2)\]
which is nonzero by Corollary \ref{cor:E75}. Then $p_1(M)=\al_3^*\al_2^*\>p_1(\iota)\neq0$ by Proposition \ref{prop:step3}. 
\end{proof}

\begin{thm}
Let $G=E_{7(-25)}$ and let $\Ga\sbs G$ be a cocompact lattice. Then $p_1(\Ga\bs G/K)\neq0$. 
\end{thm}
\begin{proof}
From Equation (\ref{eqn:chernE725}), one computes % for the total Chern class of the isotropy representation, one computes
\[p_1(\iota)=6 (y_1^2 +\cd+ y_5^2) + 2 y_6^2 + 3 y_7^2.\]
From Section \ref{sec:cwE7}, the kernel of
\[\al_2^*:H^4\big(BE_{7(-25)}\big)\ra H^4\big(B(E_{7(-25)})^\de\big)\]
is generated by 
\[6 (y_1^2 +  y_2^2 + y_3^2 + y_4^2 + y_5^2) + 2 y_6^2 + y_7^2.\]
Then 
\[\al_2^*\>p_1(\iota)=2\>\al_2^*(y_7^2).\]
By Corollary \ref{cor:E725}, $\al_2^*(y_7^2)\neq0$, and so $p_1(M)=\al_3^*\al_2^*\>p_1(\iota)\neq0$ by Proposition \ref{prop:step3}. 
\end{proof}

\subsection{Pontryagin classes for $E_6$-manifolds}

\begin{thm}
Let $G=E_{6(6)}$ and let $\Ga\sbs G$ be a cocompact lattice. Then $p_1(\Ga\bs G/K)=0$ and $p_2(\Ga\bs G/K)\neq0$. 
\end{thm}
\begin{proof}
From Equation (\ref{eqn:chernE66}), one computes %for the total Chern class of the isotropy representation, one computes 
\[p_1(\iota)=14(y_1^2+y_2^2+y_3^2+y_4^2)\]
and 
\[p_2(\iota)=91\left(\sum y_i^4\right)+166\left(\sum y_i^2y_j^2\right).\]
From Section \ref{sec:cwE6}, $I=y_1^2+y_2^2+y_3^2+y_4^2$ and $I^2$ generate the kernel of 
\[\al_2^*:H^k\big(BE_{6(6)}\big)\ra H^k\big(B(E_{6(-14)})^\de\big)\]
for $k=4,8$. Then $\al_2^*\>p_1(\iota)=0$ and 
\[\al_2^*\>p_2(\iota)=8\>\al_2^*\left(\sum y_i^4\right)+83\>\underbrace{\al_2^*\left(\sum y_i^4+2\sum y_i^2y_j^2\right)}_{=0}.
\]
By Corollary \ref{cor:E66}, $\al_2^*\left(\sum y_i^4\right)\neq0$, and so $p_2(M)=\al_3^*\al_2^*\>p_2(\iota)\neq0$ by Proposition \ref{prop:step3}.
\end{proof}

\begin{thm}
Let $G=E_{6(2)}$ and let $\Ga\sbs G$ be a cocompact lattice. Then $p_1(\Ga\bs G/K)\neq0$. 
\end{thm}
\begin{proof}
From Equation (\ref{eqn:chernE62}), one computes % for the total Chern class for the isotropy representation, one computes
\[p_1(\iota)=10\big(\sum_{i=1}^6y_i^2\big)+20\>y_7^2+8\big(\sum_{1\le i<j\le 6} y_iy_j\big)=6\big(\sum_{i=1}^6y_i^2\big)+20\>y_7^2.\]
In Section \ref{sec:cwE6}, we computed that
\[I=\sum_{i=1}^6 y_i^2+2y_7^2\]
generates (as a vector space) the kernel of 
\[\al_2^*:H^4\big(BE_{6(2)}\big)\ra H^4\big(B(E_{6(2)})^\de\big).\]
Then 
\[\al_2^*\>p_1(\iota)=\al_2^*(6I+8\>y_7^2)=8\>\al_2^*(y_7^2).\]
By Corollary \ref{cor:E62}, $\al_2^*(y_7^2)\neq0$. Then $p_1(M)=\al_3^*\al_2^*\>p_1(M)\neq0$ by Proposition \ref{prop:step3}. 
\end{proof}

\begin{thm}
Let $G=E_{6(-14)}$ and let $\Ga\sbs G$ be a cocompact lattice. Then $p_1(\Ga\bs G/K)\neq0$. 
\end{thm}
\begin{proof}
From Equation (\ref{eqn:chernE614}), one computes % for the total Chern class of the isotropy representation, one computes
\[p_1(\iota)=4 (y_1^2 + y_2^2 + y_3^2 + y_4^2 + y_5^2) + 144\> y_6^2.\]
By Theorem \ref{thm:milnor} and the computation of Section \ref{sec:cwE6},
\[I=3(y_1^2+\cd+y_5^2)+y_6^2\]
generates the kernel of 
\[\al_2^*:H^4\big(BE_{6(-14)}\big)\ra H^4\big(B(E_{6(-14)})^\de\big).\]
Since $p_1(\iota)\neq c\cdot I$ for any scalar $c$, we conclude that $p_1(\iota)$ is not in $\ker\al_2^*$. Hence $p_1(M)=\al_3^*\al_2^*\>p_1(\iota)\neq0$ by Proposition \ref{prop:step3}. 
\end{proof}

\begin{thm}
Let $G=E_{6(-26)}$ and let $\Ga\sbs G$ be a cocompact lattice. Then the $p_i(\Ga\bs G/K)=0$ for $i>0$.
\end{thm}
\begin{proof}
By Corollary \ref{cor:E626}, $\al_2^*$ is zero in positive degrees. 
\end{proof}

\subsection{Pontryagin classes for $F_4$-manifolds}

\begin{thm}
Let $G=F_{4(4)}$ and let $\Ga\sbs G$ be a cocompact lattice. Then $p_1(\Ga\bs G/K)\neq0$.
\end{thm}

\begin{proof}
From Equation (\ref{eqn:chernF44}), one computes % for the total Chern class of the isotropy representation, one computes 
\[p_1(\iota)=14y_1^2+10y_2^2+10y_3^2+10y_4^2.\]
By Theorem \ref{thm:milnor} and the computation in Section \ref{sec:cwF4}, $\sum_{i=1}^4y_i^2$ generates the kernel of 
\[\al_2^*:H^4(BF_{4(-20)})\ra H^4(BF_{4(-20)}^{\de}).\]
Then 
\[\al_2^*\>p_1(\iota)=4\>\al_2^*(y_1^2)+10\>\underbrace{\al_2^*(y_1^2+y_2^2+y_3^2+y_4^2)}_{=0}.\]By Corollary \ref{cor:F44}, $\al_2^*\>p_1(\iota)=4\al_2^*(y_1^2)\neq0$, so $p_1(M)=\al_3^*\al_2^*\>p_1(\iota)\neq0$ by Proposition \ref{prop:step3}. 
\end{proof}

\begin{thm}
Let $G=F_{4(-20)}$ and let $\Ga\sbs G$ be a cocompact lattice. Then $p_1(\Ga\bs G/K)=0$ and $p_2(\Ga\bs G/K)\neq0$. 
\end{thm}
\begin{proof}
From Equation (\ref{eqn:chernF420}), one computes %for the total Chern class of the isotropy representation, one computes
\[p_1(\iota)=2(y_1^2+\cd+y_4^2)\>\>\>\>\text{ and }\>\>\>\>p_2(\iota)=\fr{7}{4}\left(\sum y_i^4\right)+\fr{5}{2}\left(\sum_{1\le i<j\le4} y_i^2y_j^2\right).\]
By Theorem \ref{thm:milnor} and the computation in Section \ref{sec:cwF4}, $\sum_{i=1}^4 y_i^2$ and $\left(\sum_{i=1}^4y_i^2\right)^2=\sum y_i^4+2\sum y_i^2y_j^2$ generate the kernel of 
\[\al_2^*:H^k(BF_{4(-20)})\ra H^k(BF_{4(-20)}^{\de})\]
for $k=4,8$, respectively. Then $\al_2^*\>p_1(\iota)=0$ and 
\[\al_2^*\>p_2(\iota)=\fr{7}{4}\underbrace{\al_2^*\left(\sum y_i^4+2\sum y_i^2y_j^2\right)}_{=0}-\fr{2}{2}\al_2^*\left(\sum y_i^2y_j^2\right)=-\al_2^*\left(\sum y_i^2y_j^2\right).\]
Then $\al_2^*\>p_2(\iota)\neq0$ by Corollary \ref{cor:F420}, and so $p_2(M)=\al_3^*\al_2^*\>p_2(\iota)\neq0$ by Proposition \ref{prop:step3}. 
\end{proof}

\subsection{Pontryagin classes for $G_2$-manifolds}

\begin{thm}
Let $G=G_{2(2)}$, and let $\Ga\sbs G$ be a cocompact lattice. Then $p_1(\Ga\bs G/K)\neq0$. 
\end{thm}
\begin{proof}
From Equation (\ref{eqn:chernG22}), one computes % for the total Chern class of the isotropy representation, one computes $p_1(\iota)=20 y_1^2 + 4 y_2^2$. By Theorem \ref{thm:milnor} and the computation is Section \ref{sec:cwG2}, $3y_1^2+y_2^2$ generates the kernel of 
\[\al_2^*: H^4(BG_{2(2)})\ra H^4(BG_{2(2)}^\de)\]
Then 
\[\al_2^*\>p_1(\iota)= \al_2^*\big(8y_1^2+12y_1^2+4y_2^2\big)=8\>\al_2^*(y_1^2)+4\>\underbrace{\al_2^*(3y_1^2+y_2^2)}_{=0}=8\>\al_2^*(y_1^2).\]
By Corollary \ref{cor:G2}, $\al_2^*(y_1^2)\neq0$, and so $p_1(M)=\al_3^*\al_2^*\>p_1(\iota)\neq0$ by Proposition \ref{prop:step3}. 
\end{proof}

\bibliographystyle{plain}
\bibliography{/Users/sm/Documents/GeneralMath/LaTeX/tpbib}
\end{document}